\documentclass[journal]{IEEEtran}
\usepackage[cmex10]{amsmath}
\usepackage{amssymb,amsthm}

\newcommand{\bbC}{\mathbb{C}}
\newcommand{\bbF}{\mathbb{F}}
\newcommand{\bbH}{\mathbb{H}}

\newcommand{\bbR}{\mathbb{R}}
\newcommand{\bbT}{\mathbb{T}}
\newcommand{\bbZ}{\mathbb{Z}}

\newcommand{\bfA}{\mathbf{A}}

\newcommand{\bfG}{\mathbf{G}}
\newcommand{\bfI}{\mathbf{I}}
\newcommand{\bfJ}{\mathbf{J}}
\newcommand{\bfL}{\mathbf{L}}

\newcommand{\bfx}{\mathbf{x}}
\newcommand{\bfy}{\mathbf{y}}

\newcommand{\bfphi}{\boldsymbol{\varphi}}
\newcommand{\bfPhi}{\boldsymbol{\Phi}}
\newcommand{\bfpsi}{\boldsymbol{\psi}}
\newcommand{\bfPsi}{\boldsymbol{\Psi}}
\newcommand{\bfchi}{\boldsymbol{\chi}}

\newcommand{\bfdelta}{\boldsymbol{\delta}}
\newcommand{\bfDelta}{\boldsymbol{\Delta}}
\newcommand{\bfGamma}{\boldsymbol{\Gamma}}

\newcommand{\calD}{\mathcal{D}}
\newcommand{\calE}{\mathcal{E}}
\newcommand{\calG}{\mathcal{G}}
\newcommand{\calH}{\mathcal{H}}
\newcommand{\calK}{\mathcal{K}}
\newcommand{\calM}{\mathcal{M}}
\newcommand{\calN}{\mathcal{N}}

\newcommand{\rmc}{\mathrm{c}}

\newcommand{\rme}{\mathrm{e}}
\newcommand{\rmi}{\mathrm{i}}

\newcommand{\tr}{\operatorname{tr}}
\newcommand{\Tr}{\operatorname{Tr}}
\newcommand{\coh}{\operatorname{coh}}
\newcommand{\ETF}{\operatorname{ETF}}
\newcommand{\Fro}{{\operatorname{Fro}}}

\newcommand{\RDS}{\operatorname{RDS}}
\newcommand{\BIBD}{\operatorname{BIBD}}

\newcommand{\rank}{\operatorname{rank}}
\newcommand{\Span}{\operatorname{span}}
\newcommand{\MUETF}{\operatorname{MUETF}}

\newcommand{\abs}[1]{|{#1}|}

\newcommand{\biggparen}[1]{\biggl({#1}\biggr)}

\newcommand{\bigbracket}[1]{\bigl[{#1}\bigr]}

\newcommand{\set}[1]{\{{#1}\}}

\newcommand{\biggset}[1]{\biggl\{{#1}\biggr\}}

\newcommand{\norm}[1]{\|{#1}\|}

\newcommand{\ip}[2]{\langle{#1},{#2}\rangle}

\newtheorem{theorem}{Theorem}

\newtheorem{corollary}{Corollary}

\theoremstyle{definition}
\newtheorem{remark}{Remark}
\newtheorem{example}{Example}
\newtheorem{definition}{Definition}

\begin{document}
\title{Mutually Unbiased Equiangular Tight Frames}

\author{Matthew~Fickus,~\IEEEmembership{Senior Member,~IEEE}, Benjamin~R.~Mayo%
\thanks{M.~Fickus and B.~R.~Mayo are with the Department of Mathematics and Statistics, Air Force Institute of Technology, Wright-Patterson Air Force Base, OH 45433, USA, e-mail: Matthew.Fickus@afit.edu.}}

\maketitle

\begin{abstract}
An equiangular tight frame (ETF) yields a type of optimal packing of lines in a Euclidean space.
ETFs seem to be rare, and all known infinite families of them arise from some type of combinatorial design.
In this paper, we introduce a new method for constructing ETFs.
We begin by showing that it is sometimes possible to construct multiple ETFs for the same space that are ``mutually unbiased" in a way that is analogous to the quantum-information-theoretic concept of mutually unbiased bases.
We then show that taking certain tensor products of these mutually unbiased ETFs with other ETFs sometimes yields infinite families of new complex ETFs.
\end{abstract}


\section{Introduction}

For any $N\geq D\geq1$, $N>1$, Welch~\cite{Welch74} gives the following bound on the \textit{coherence} of $N$ unit vectors $\set{\bfphi_n}_{n=1}^{N}$ in $\bbC^D$:
\begin{equation}
\label{eq.Welch}
\coh(\set{\bfphi_n}_{n=1}^{N})
:=\max_{n\neq n'}\abs{\ip{\bfphi_n}{\bfphi_{n'}}}
\geq\bigbracket{\tfrac{N-D}{D(N-1)}}^{\frac12}.
\end{equation}
It is well known~\cite{StrohmerH03} that $\set{\bfphi_n}_{n=1}^{N}$ achieves equality in \eqref{eq.Welch} if and only if $\set{\bfphi_n}_{n=1}^{N}$ is an \textit{equiangular tight frame} (ETF) for $\bbC^D$,
that is, if and only if the value of $\abs{\ip{\bfphi_n}{\bfphi_{n'}}}$ is constant over all $n\neq n'$ (equiangularity) and there exists $C>0$ such that
C$\norm{\bfy}^2=\sum_{n=1}^{N}\abs{\ip{\bfphi_n}{\bfy}}^2$ for all $\bfy\in\bbC^D$ (tightness).

The coherence of any unit vectors is the cosine of the smallest principal angle between any two of the lines (one-dimensional subspaces) they individually span.
By achieving equality in~\eqref{eq.Welch},
an ETF yields $N$ lines in $\bbC^D$ whose smallest pairwise principal angle is as large as possible,
namely an optimal way to pack $N$ points on the projective space that consists of all lines in $\bbC^D$.
Due to their optimality,
ETFs arise in various applications including waveform design for wireless communication~\cite{StrohmerH03},
compressed sensing~\cite{BajwaCM12,BandeiraFMW13},
quantum information theory~\cite{Zauner99,RenesBSC04}
and algebraic coding theory~\cite{JasperMF14}.

Much of the ETF literature is devoted to the \textit{existence problem}:
for what $D$ and $N$ does there exist an $\ETF(D,N)$, that is, an $N$-vector ETF for $\bbC^D$?
Here, one key subproblem is to resolve \textit{Zauner's conjecture} that an $\ETF(D,D^2)$ exists for any $D\geq1$~\cite{Zauner99,RenesBSC04}.
In quantum information theory, such an ETF is called a \textit{symmetric, informationally complete, positive operator-valued measure} (SIC-POVM),
and a finite, but remarkable number of these have already been found~\cite{FuchsHS17}.
Another key subproblem is to characterize the existence of \textit{real} $\ETF(D,N)$,
that is, ETFs where $\ip{\bfphi_n}{\bfphi_{n'}}\in\bbR$ for all $n,n'$.
Real ETFs equate to a subclass of \textit{strongly regular graphs} (SRGs)~\cite{vanLintS66,Seidel76,HolmesP04,Waldron09},
which are a mature subject in and of themselves~\cite{Brouwer07,Brouwer17,CorneilM91}.
In general, the existence problem remains poorly understood,
with lists of known ETFs~\cite{FickusM16} falling far short of known necessary conditions, namely that an $\ETF(D,N)$ with $1<D<N-1$ can only exist if $N\leq\min\set{D^2,(N-D)^2}$~\cite{HolmesP04} (a generalization of \textit{Gerzon's bound}~\cite{LemmensS73}) and that an $\ETF(3,8)$ does not exist~\cite{Szollosi14}.

All known positive existence results for $\ETF(D,N)$ with $1<D<N-1$ are due to explicit construction involving some type of combinatorial design; see~\cite{FickusM16} for a survey.
ETFs whose \textit{redundancy} $\frac ND$ is either nearly or exactly $2$ arise from the related concepts of Hadamard matrices, conference matrices, Gauss sums and Paley tournaments~\cite{StrohmerH03,HolmesP04,Renes07,Strohmer08}.
The equivalence between real ETFs and certain SRGs has been partially generalized to the complex case using roots of unity~\cite{BodmannPT09,BodmannE10},
abelian distance-regular antipodal covers of complete graphs~\cite{CoutinkhoGSZ16,FickusJMPW19},
and association schemes~\cite{IversonJM16}.
\textit{Harmonic ETFs} equate to \textit{difference sets} in finite abelian groups~\cite{Konig99,XiaZG05,DingF07}.
\textit{Steiner ETFs} arise from \textit{balanced incomplete block designs} (BIBDs)~\cite{GoethalsS70,FickusMT12}.
Nontrivial generalizations of the Steiner ETF construction yield other ETFs arising from projective planes containing hyperovals~\cite{FickusMJ16},
Steiner triple systems~\cite{FickusJMP18},
and group divisible designs~\cite{FickusJ19}.

In this paper, we provide a new method for constructing ETFs.
It is inspired by an ETF-based perspective~\cite{FickusJKM18,FickusS20} of a classical factorization~\cite{GordonMW62} of the complement of a \textit{Singer} difference set in terms of a \textit{relative difference set} (RDS).
The main idea is to take tensor products of vectors in a given $\ETF(D_1,N_1)$ with those belonging to a collection of $N_1$ distinct $\ETF(D_2,N_2)$ that are \textit{mutually unbiased} in the quantum-information-theoretic sense.
We show that this technique, for example, yields (complex) $\ETF(D,N)$ with
\begin{equation}
\label{eq.new ETF from 2 pos}
D=\tfrac{Q-1}{Q+1}(\tfrac{Q-1}{2}Q^{2J-1}-1),
\quad
N=\tfrac{Q-1}{Q+1}(Q^{2J}-1),
\end{equation}
for any prime power $Q\geq4$ and $J\geq2$,
as well as ones with
\begin{equation}
\label{eq.neq ETF from GQ}
D=\tfrac{Q^3+1}{Q^4-1}(\tfrac{Q^3+1}{Q+1}Q^{4J-3}-1),
\
N=\tfrac{Q^3+1}{Q^4-1}(Q^{4J}-1),
\end{equation}
for any prime power $Q\geq2$ and $J\geq2$.
Remarkably, all such ETFs seem to be new.
For example, taking $J=2$ and $Q=4,5$ in~\eqref{eq.new ETF from 2 pos} yields $\ETF(57,153)$ and $\ETF(166,416)$, neither of which were previously known~\cite{FickusM16}.

In the next section we establish notation,
and review known concepts that we will need later on.
In Section~\ref{sec.MUETFs},
we explain what it means for several $\ETF(D,N)$ to be mutually unbiased (Definition~\ref{def.MUETF}),
and give a necessary condition on their existence (Theorem~\ref{thm.Gerzon}).
We moreover construct mutually unbiased ETFs from RDSs,
both in general (Theorem~\ref{thm.RDS gives MUETF})
and using a classical family (Corollary~\ref{cor.Singer MUETFF}).
In Section~\ref{sec.new ETFs},
we discuss the aforementioned tensor-product-based technique (Theorem~\ref{thm.tensor}).
Combining it with Corollary~\ref{cor.Singer MUETFF} yields our main result (Theorem~\ref{thm.main result}).
In special cases where the initial $\ETF(D_1,N_1)$ is \textit{positive} or \textit{negative} in the sense of~\cite{FickusJ19},
our main result yields several infinite families of new (complex) ETFs (Corollary~\ref{cor.pos neg}).

\section{Background}
\label{sec.background}

\subsection{Equiangular tight frames and Naimark complements}
Let $\bbF$ be either $\bbR$ or $\bbC$.
For any $N$-element set of indices $\calN$,
equip $\bbF^\calN:=\set{\bfx:\calN\rightarrow\bbF}$
with the complex dot product
$\ip{\bfx_1}{\bfx_2}:=\sum_{n\in\calN}[\bfx_1(n)]^*\bfx_2(n)$ which,
like all inner products in this paper, is conjugate-linear in its first argument.
For any finite sequence $\set{\bfphi_n}_{n\in\calN}$ of vectors in a Hilbert space $\bbH$ over $\bbF$,
the corresponding \textit{synthesis operator} is $\bfPhi:\bbF^\calN\rightarrow\bbH$,
$\bfPhi\bfx:=\sum_{n\in\calN}\bfx(n)\bfphi_n$.
Its adjoint $\bfPhi^*:\bbH\rightarrow\bbF^\calN$,
$(\bfPhi^*\bfy)(n)=\ip{\bfphi_n}{\bfy}$ is called the \textit{analysis operator}.
We sometimes identify a vector $\bfphi\in\bbH$ with its synthesis operator $\bfphi:\bbF\rightarrow\bbH$, $\bfphi(x):=x\bfphi$,
an operator whose adjoint is the linear functional $\bfphi^*:\bbH\rightarrow\bbF$,
$\bfphi^*\bfy=\ip{\bfphi}{\bfy}$.
In the special case where $\bbH=\bbF^\calD$ for some $D$-element set $\calD$,
$\bfPhi$ is just the $\calD\times\calN$ matrix whose $n$th column is $\bfphi_n$,
and $\bfPhi^*$ is its $\calN\times\calD$ conjugate-transpose.

In general,
the \textit{frame operator} of $\set{\bfphi_n}_{n\in\calN}$ is the composition
$\bfPhi\bfPhi^*:\bbH\rightarrow\bbH$ of its synthesis and analysis operators,
namely
$\bfPhi\bfPhi^*=\sum_{n\in\calN}\bfphi_n^{}\bfphi_n^*$,
$\bfPhi\bfPhi^*\bfy=\sum_{n\in\calN}\ip{\bfphi_n}{\bfy}\bfphi_n$.
The reverse composition is the $\calN\times\calN$ \textit{Gram matrix} that has
$(\bfPhi^*\bfPhi)(n,n')=\ip{\bfphi_n}{\bfphi_{n'}}$ as its $(n,n')$th entry.
This matrix has $\rank(\bfPhi^*\bfPhi)=\rank(\bfPhi)=\dim(\Span\set{\bfphi_n}_{n\in\calN})$ and is positive-semidefinite.
Conversely, any positive-semidefinite $\calN\times\calN$ matrix $\bfG$ factors as $\bfG=\bfPhi^*\bfPhi$ where $\bfPhi$ is the synthesis operator of a sequence  $\set{\bfphi_n}_{n\in\calN}$ that spans $\bbH$ where $\dim(\bbH)=\rank(\bfG)$.
Here, $\set{\bfphi_n}_{n\in\calN}$ and $\bbH$ are only unique up to unitary transformations,
meaning we can take $\bbH=\bbF^D$ if so desired, where $D=\rank(\bfG)$.

We say $\set{\bfphi_n}_{n\in\calN}$ is a ($C$-)\textit{tight frame} for $\bbH$ if $\bfPhi\bfPhi^*=C\bfI$ for some $C>0$.
By the polarization identity, this equates to having $\sum_{n\in\calN}\abs{\ip{\bfphi_n}{\bfy}}^2=\norm{\bfPhi^*\bfy}^2=C\norm{\bfy}^2$ for all $\bfy\in\bbH$.
An $\calN\times\calN$ self-adjoint matrix $\bfG$ is the Gram matrix $\bfPhi^*\bfPhi$ of a $C$-tight frame $\set{\bfphi_n}_{n\in\calN}$ for some space $\bbH$ if and only if $\bfG^2=C\bfG$,
namely when $\frac1C\bfG$ is an orthogonal projection operator.
In particular, $\Tr(\bfG)=CD$ where $D=\rank(\bfG)$.

We say $\set{\bfphi_n}_{n\in\calN}$ is a \textit{unit norm tight frame} (UNTF) for $\bbH$ if it is a tight frame for $\bbH$ and $\norm{\bfphi_n}=1$ for all $n$.
Here, we necessarily have $N=\Tr(\bfPhi^*\bfPhi)=CD$ where $D=\dim(\bbH)$.
As such, a sequence $\set{\bfphi_n}_{n\in\calN}$ of $N$ unit vectors in $\bbH$ is a UNTF for $\bbH$ if and only if
\begin{equation*}
\norm{\bfPhi\bfPhi^*-\tfrac ND\bfI}_{\Fro}^2
=\Tr[(\bfPhi\bfPhi^*-\tfrac ND\bfI)^2]
=\norm{\bfPhi^*\bfPhi}_{\Fro}^2-\tfrac{N^2}{D}
\end{equation*}
is zero.
That is, any $N$ unit vectors $\set{\bfphi_n}_{n\in\calN}$ in $\bbH$ satisfy
\begin{equation}
\label{eq.FP}
\tfrac{N^2}{D}
\leq\norm{\bfPhi^*\bfPhi}_{\Fro}^2
=\sum_{n\in\calN}\sum_{n'\in\calN}\abs{\ip{\bfphi_n}{\bfphi_{n'}}}^2,
\end{equation}
and achieve equality here if and only if they form a UNTF for $\bbH$.
When $\set{\bfphi_n}_{n\in\calN}$ is a UNTF for $\bbH$,
$\frac DN\bfPhi^*\bfPhi$ is an orthogonal projection operator with constant diagonal entries,
implying $\bfI-\frac DN\bfPhi^*\bfPhi$ is another such operator of rank $N-D$.
In particular, when this occurs with $N>D$,
$\frac{N}{N-D}\bfI-\frac{D}{N-D}\bfPhi^*\bfPhi$ is the Gram matrix
\smash{$\tilde{\bfPhi}^*\tilde{\bfPhi}$} of a UNTF
\smash{$\set{\tilde{\bfphi}_n}_{n\in\calN}$} for a space \smash{$\tilde{\bbH}$}
of dimension $N-D$.
Such a sequence \smash{$\set{\tilde{\bfphi}_n}_{n\in\calN}$} is called a \textit{Naimark complement} of $\set{\bfphi_n}_{n\in\calN}$.
Up to unitary transformations, it is uniquely defined according to
\begin{equation}
\label{eq.Naimark}
\ip{\tilde{\bfphi}_n}{\tilde{\bfphi}_{n'}}
=\left\{\begin{array}{cl}
1,&\ n=n',\\
-\tfrac{D}{N-D}\ip{\bfphi_n}{\bfphi_{n'}},&\ n\neq n'.\end{array}\right.
\end{equation}

Returning to \eqref{eq.FP}, we now note that bounding the off-diagonal terms of this sum by their maximum value gives
\begin{align}
\label{eq.Welch derivation 1}
\tfrac{N^2}{D}
&\leq\sum_{n\in\calN}\sum_{n'\in\calN}\abs{\ip{\bfphi_n}{\bfphi_{n'}}}^2\\
\label{eq.Welch derivation 2}
&\leq N+N(N-1)\max_{n\neq n'}\abs{\ip{\bfphi_n}{\bfphi_{n'}}}^2,
\end{align}
which equates to the Welch bound~\eqref{eq.Welch}.
Moreover, equality in~\eqref{eq.Welch} is equivalent to equality in both~\eqref{eq.Welch derivation 1} and~\eqref{eq.Welch derivation 2},
namely to when $\set{\bfphi_n}_{n\in\calN}$ is a UNTF for $\bbH$ that also happens to be equiangular, namely an ETF for $\bbH$.
In particular, equality holds in~\eqref{eq.Welch} if and only if $\abs{\ip{\bfphi_n}{\bfphi_{n'}}}^2=\frac{N-D}{D(N-1)}$ for all $n\neq n'$,
and in this case, $\set{\bfphi_n}_{n\in\calN}$ is necessarily a UNTF for $\bbH$.
By \eqref{eq.Naimark}, the Naimark complement of an $\ETF(D,N)$ is an $\ETF(N-D,N)$, a fact we will use often.

\subsection{Harmonic frames and relative difference sets}

A \textit{character} of a finite abelian group $\calG$ is a homomorphism
$\gamma:\calG\rightarrow\bbT:=\set{z\in\bbC: \abs{z}=1}$.
The set of all such characters is known as the \textit{(Pontryagin) dual} $\hat{\calG}$ of $\calG$, which is itself a group under pointwise multiplication.
In fact, since $\calG$ is finite,
$\hat{\calG}$ is known to be isomorphic to $\calG$.
The synthesis operator $\bfGamma$ of the characters of $\calG$ is a square $\calG\times\hat{\calG}$ matrix having $\bfGamma(g,\gamma)=\gamma(g)$ for all $g$ and $\gamma$.
$\bfGamma$ is often called the \textit{character table} of $\calG$,
and its adjoint $\bfGamma^*:\bbC^{\calG}\rightarrow\bbC^{\hat{\calG}}$,
$(\bfGamma^*\bfy)(\gamma)=\ip{\gamma}{\bfy}$ (the analysis operator of the characters) is the \textit{discrete Fourier transform} (DFT) over $\calG$.
Since $\calG$ is finite, it is known that its characters form an equal-norm orthogonal basis for $\bbC^\calG$,
and so $\bfGamma^{-1}=\frac1G\bfGamma^*$ where $G=\#(\calG)$.
In particular $\bfGamma\bfGamma^*=G\bfI$.

For any $\calD\subseteq\calG$ with $D=\#(\calD)>0$,
let $\bfPsi$ be the synthesis operator of the corresponding \textit{harmonic frame} $\set{\bfpsi_\gamma}_{\gamma\in\hat{\calG}}$, that is, the normalized restrictions of the characters of $\calG$ to $\calD$:
\begin{equation}
\label{eq.harmonic frame}
\bfPsi\in\bbC^{\calD\times\hat{\calG}},
\quad
\bfPsi(d,\gamma)
=\bfpsi_\gamma(d)
:=D^{-\frac12}\gamma(d).
\end{equation}
Any such frame is automatically a UNTF:
for any $d_1,d_2\in\calD$,
\begin{equation*}
(\bfPsi\bfPsi^*)(d_1,d_2)
=\tfrac1D\sum_{\gamma\in\hat{\calG}}[\gamma(d_1)]^*\gamma(d_2)
=\tfrac1D(\bfGamma\bfGamma^*)(d_1,d_2),
\end{equation*}
and so $\bfPsi\bfPsi^*=\tfrac1D\bfGamma\bfGamma^*=\tfrac GD\bfI$.
Meanwhile, for any $\gamma_1,\gamma_2\in\hat{\calG}$,
the corresponding entry of the Gram matrix is
\begin{align*}
\ip{\bfpsi_{\gamma_1}}{\bfpsi_{\gamma_2}}
&=\tfrac1D\sum_{d\in\calD}[\gamma_1(d)]^*\gamma_2(d)\\
&=\tfrac1D\sum_{g\in\calG}[(\gamma_1^{}\gamma_2^{-1}(g)]^*\bfchi_\calD(g)\\
&=\tfrac1D(\bfGamma^*\bfchi_\calD)(\gamma_1^{}\gamma_2^{-1}),
\end{align*}
where $\bfchi_\calD\in\bbC^\calG$ is the $\set{0,1}$-valued characteristic (indicator) function of $\calD$.
In particular,
\begin{equation}
\label{eq.RDS 0}
\abs{\ip{\bfpsi_{\gamma_1}}{\bfpsi_{\gamma_2}}}^2
=\tfrac1{D^2}\abs{(\bfGamma^*\bfchi_\calD)(\gamma_1^{}\gamma_2^{-1})}^2,
\quad\forall\,\gamma_1,\gamma_2\in\hat{\calG}.
\end{equation}
To continue,
we exploit the fact that the DFT $\bfGamma^*$ distributes over \textit{convolution}:
for any $\bfy_1,\bfy_2\in\bbC^\calG$, defining $\bfy_1*\bfy_2\in\bbC^\calG$ by $(\bfy_1*\bfy_2)(g):=\sum_{g'\in\calG}\bfy_1(g')\bfy_2(g-g')$,
we have
\begin{equation*}
[\bfGamma^*(\bfy_1*\bfy_2)](\gamma)
=(\bfGamma^*\bfy_1)(\gamma)(\bfGamma^*\bfy_2)(\gamma),
\quad\forall\,\gamma\in\hat{\calG}.
\end{equation*}
(When considering such $\calG$ in general, we default to writing the group operation on $\calG$ and its dual $\hat{\calG}$ as addition and multiplication, respectively.)
Meanwhile, the DFT of the \textit{involution} $\tilde{\bfy}\in\bbC^\calG$ of $\bfy\in\bbC^\calG$, $\tilde{\bfy}(g):=[\bfy(-g)]^*$ is
$(\bfGamma^*\tilde{\bfy})(\gamma)=[(\bfGamma^*\bfy)(\gamma)]^*$ for all $\gamma\in\hat{\calG}$.
Combined, we have
\begin{equation*}
\abs{(\bfGamma^*\bfchi_\calD)(\gamma)}^2
=[\bfGamma^*(\bfchi_\calD*\tilde{\bfchi}_\calD)](\gamma),
\quad\forall\,\gamma\in\hat{\calG}.
\end{equation*}
Here, $\bfchi_\calD*\tilde{\bfchi}_\calD$ is the \textit{autocorrelation} of $\bfchi_\calD$,
which counts the number of distinct ways that any given $g\in\calG$ can be written as a difference of members of $\calD$:
\begin{align*}
(\bfchi_\calD*\tilde{\bfchi}_\calD)(g)
&=\sum_{g'\in\calG}\bfchi_\calD(g')\tilde{\bfchi}_\calD(g-g')\\
&=\sum_{g'\in\calG}\bfchi_\calD(g')\bfchi_{g+\calD}(g')\\
&=\#\set{\calD\cap(g+\calD)}\\
&=\#\set{(d,d')\in\calD\times\calD: g=d-d'}.
\end{align*}
Altogether, we see that there is a relationship between the combinatorial properties of the differences $d-d'$ of members of $\calD$ and the magnitudes of the inner products of vectors that belong to the corresponding harmonic frame.
This relationship has long been exploited~\cite{Turyn65} to characterize certain types of $\calD$ including, as we now explain, relative difference sets:

\begin{definition}
\label{def.RDS}
Let $\calH$ be an $H$-element subgroup of an abelian group $\calG$ of order $G$.
A $D$-element subset $\calD$ of $\calG$ is an $\calH$-RDS for $\calG$
if there exists a constant $\Lambda$ such that
\begin{equation}
\label{eq.RDS 1}
\bfchi_\calD*\tilde{\bfchi}_\calD
=\Lambda(\bfchi_\calG-\bfchi_\calH)+D\bfdelta_0,
\end{equation}
namely if no nonzero member of $\calH$ is a difference of two members of $\calD$
while every member of $\calH^\rmc$ can be written as a difference of members of $\calD$ in exactly $\Lambda$ ways.
\end{definition}

In the literature, an RDS with these parameters is usually denoted as an ``$\RDS(N,H,D,\Lambda)$" where $N=\frac{G}{H}$.
In the special case where $\calH=\set{0}$, an $\calH$-RDS for $\calG$ is simply called a \textit{difference set} for $\calG$.
To proceed, we use the \textit{Poisson summation formula},
namely that $\bfGamma^*\bfchi_\calH=H\bfchi_{\calH^\perp}$
where $\calH^\perp:=\set{\gamma\in\hat{\calG}: \gamma(h)=1,\ \forall\,h\in\calH}$ is the \textit{annihilator} of $\calH$,
which is a subgroup of $\hat{\calG}$ that is isomorphic to $\calG/\calH$.
(``The DFT of a comb is a comb.")
In particular,
taking the DFT of~\eqref{eq.RDS 1} gives that $\calD$ is an $\calH$-RDS for $\calG$ if and only if
\begin{equation}
\label{eq.RDS 2}
\abs{(\bfGamma^*\bfchi_\calD)(\gamma)}^2
=\Lambda[G\bfdelta_1(\gamma)-H\bfchi_{\calH^\perp}(\gamma)]+D,
\ \forall\,\gamma\in\hat{\calG}.
\end{equation}
Here, evaluating \eqref{eq.RDS 2} at $\gamma=1$ gives
$D^2=\Lambda(G-H)+D$,
namely that $\Lambda=\frac{D(D-1)}{G-H}=\frac{D(D-1)}{H(N-1)}$;
this also follows from a simple counting argument.
As such,
\begin{equation*}
D-\Lambda H
=D-\tfrac{D(D-1)}{N-1}
=\tfrac{D(N-D)}{N-1}
\end{equation*}
and so~\eqref{eq.RDS 2} equates to having
\begin{equation*}
\abs{(\bfGamma^*\bfchi_\calD)(\gamma)}^2
=\left\{\begin{array}{cl}
\frac{D(N-D)}{N-1},&\ \gamma\in\calH^\perp,\gamma\neq1,\smallskip\\
D,&\ \gamma\notin\calH^\perp.
\end{array}\right.
\end{equation*}
In light of \eqref{eq.RDS 0},
we see that $\calD$ is an $\calH$-RDS for $\calG$ if and only if
the corresponding harmonic frame~\eqref{eq.harmonic frame} satisfies
\begin{equation}
\label{eq.RDS 3}
\abs{\ip{\bfpsi_{\gamma_1}}{\bfpsi_{\gamma_2}}}^2
=\left\{\begin{array}{cl}
\frac{N-D}{D(N-1)},&\ \gamma_1^{}\gamma_2^{-1}\in\calH^\perp,\ \gamma_1\neq\gamma_2,\smallskip\\
\frac1D,&\ \gamma_1^{}\gamma_2^{-1}\notin\calH^\perp.
\end{array}\right.
\end{equation}
In the special case where $\calH=\set{0}$,
the second condition above becomes vacuous,
and this result reduces to the equivalence between difference sets and harmonic ETFs given in~\cite{XiaZG05,DingF07}.

\subsection{Positive and negative ETFs}
\label{subsec.pos neg}

In Section~\ref{sec.new ETFs},
we show that, in certain circumstances,
one can construct an $\ETF(D_1D_2,N_1N_2)$ from an $\ETF(D_1,N_1)$ and $N_1$ mutually unbiased $\ETF(D_2,N_2)$.
It turns out that this technique applies to many distinct types of $\ETF(D_1,N_1)$, including Naimark complements of Steiner ETFs~\cite{FickusMT12} and Tremain ETFs~\cite{FickusJMP18}, as well as polyphase BIBD ETFs~\cite{FickusJMPW19}.
Here, to prevent duplication of effort,
it helps to have the following concepts from~\cite{FickusJ19},
which unite the $(D_1,N_1)$ parameters of these disparate ETFs into a common framework:

\begin{definition}
\label{def.pos neg ETFs}
For any $\ETF(D,N)$ with $N>D>1$, let
\begin{equation*}
L\in\set{1,-1},
\quad
S:=\bigbracket{\tfrac{D(N-1)}{N-D}}^{\frac12},
\quad
K:=\tfrac{NS}{D(S+L)}.
\end{equation*}
When $S,K\in\bbZ$, we say this ETF is \textit{type} $(K,L,S)$.
In this case, depending on whether $L$ is $1$ or $-1$,
we also refer to such an ETF as being ($K$-)\textit{positive} or ($K$-)\textit{negative}, respectively.
\end{definition}

We caution that some ETFs are both positive and negative:
for example, the well-known $\ETF(3,9)$ is both $2$-positive and $6$-negative, being both of type $(2,1,2)$ and $(6,-1,2)$.
Regardless,
for any $\ETF(D,N)$ of type $(K,L,S)$,
Theorem~3.1 of~\cite{FickusJ19} gives expressions for $(D,N)$ in terms of $(K,L,S)$:
\begin{align}
\label{eq.pos neg D}
D
&=\tfrac{S}{K}[S(K-1)+L]\\
\label{eq.pos neg N}
N
&=(S+L)[S(K-1)+L].
\end{align}
As summarized in~\cite{FickusJ19},
almost all currently known constructions of $\ETF(D,N)$ with $N>2D>2$ are either positive or negative,
with the only exceptions being certain SIC-POVMs, harmonic ETFs, and examples where $N=2D+1$.
As summarized in Theorems~1.2, 4.1 and 4.2 of~\cite{FickusJ19},
an ETF of type $(K,L,S)$ exists whenever either:
\begin{enumerate}
\renewcommand{\labelenumi}{(\alph{enumi})}
\item
$(K,L,S)=(1,1,S)$ where $S\geq 2$ (regular simplices);\smallskip
\item
$(K,L,S)=(K,1,S)$ and a $\BIBD(V,K,1)$ exists where $V=(K-1)S+1$ (Steiner ETFs~\cite{FickusMT12}), including:
\begin{enumerate}
\renewcommand{\labelenumii}{(\roman{enumii})}
\item
when $K=2,3,4,5$, $S\geq K$, $S\equiv 0,1\bmod K$,
\item
when $K\mid S(S-1)$ and $S$ is sufficiently large;\smallskip
\end{enumerate}
\item
$(K,L,S)=(Q,1,Q)$ where $Q$ is any prime power
(Naimark complements of polyphase BIBD ETFs~\cite{FickusJMPW19});\smallskip
\item
$(K,L,S)=(2,-1,S)$ where $S\geq 3$ ($\ETF(D,N)$ with $D=\frac12(N+\sqrt{N})$);\smallskip
\item
$(K,L,S)=(3,-1,S)$ where $S\geq2$, $S\equiv0,2\bmod 3$
(Tremain ETFs~\cite{FickusJMP18});\smallskip
\item
$(K,L,S)=(Q+1,-1,Q+1)$ where $Q$ is an even prime power
(hyperoval ETFs~\cite{FickusMJ16});\smallskip
\item
$(K,L,S)=(4,-1,S)$ where $S\equiv3\bmod8$~\cite{FickusJ19}.
\end{enumerate}
This is not a comprehensive list:
for the sake of brevity and clarity,
we have omitted some infinite families of negative ETFs that are either also (postive) Steiner ETFs or are overly technical
(see, for example, Theorems~1.2, 1.3 and 4.4 of~\cite{FickusJ19} for some additional $K$-negative ETFs with $K=4,5,6,7,10,12,15$),
as well as a finite number of positive and/or negative ETFs for which, it turns out, our theory below does not apply.

\section{Mutually Unbiased ETFs}
\label{sec.MUETFs}

Let $\calM$, $\calN$ and $\calD$ be sets of cardinality $M$, $N$ and $D$, respectively, where $M\geq1$, $N\geq D\geq1$.
For each $m\in\calM$, let $\set{\bfpsi_{m,n}}_{n\in\calN}$ be an ETF for $\bbF^\calD$ with synthesis operator $\bfPsi_m$.
When $N=D$, this equates to a collection of $M$ orthonormal bases for $\bbF^\calD$;
in quantum information theory, one says that such bases are \textit{mutually unbiased} if
\smash{$\abs{\ip{\bfpsi_{m,n}}{\bfpsi_{m',n'}}}^2=\frac1D$} whenever $m\neq m'$.
As we now explain, this same condition in general ensures that the concatenation \smash{$\set{\bfpsi_{m,n}}_{m\in\calM,n\in\calN}$} of these ETFs has minimal coherence.
This concatenation is an $MN$-vector UNTF for $\bbF^\calD$,
as its synthesis operator $\bfPsi$ satisfies
\begin{equation*}
\bfPsi\bfPsi^*
=\sum_{m\in\calM}\sum_{n\in\calN}\bfpsi_{m,n}^{}\bfpsi_{m,n}^*
=\sum_{m\in\calM}\bfPsi_m^{}\bfPsi_m^*
=\tfrac{MN}{D}\bfI.
\end{equation*}
Thus~\eqref{eq.FP} gives \smash{$\norm{\bfPsi^*\bfPsi}_{\Fro}^2=\frac{M^2N^2}{D}$}.
Here, $\bfPsi^*\bfPsi$ is an $\calM\times\calM$ block matrix whose $(m,m')$th block
is the $\calN\times\calN$ \textit{cross-Gram} matrix $\bfPsi_m^*\bfPsi_{m'}^{}$ whose $(n,n')$th entry is
$(\bfPsi_m^*\bfPsi_{m'}^{})(n,n')=\ip{\bfpsi_{m,n}}{\bfpsi_{m',n'}}$.
Thus,
\begin{equation*}
\tfrac{M^2N^2}{D}
=\norm{\bfPsi^*\bfPsi}_{\Fro}^2
=\sum_{m\in\calM}\sum_{m'\in\calM}\norm{\bfPsi_m^*\bfPsi_{m'}^{}}_\Fro^2.
\end{equation*}
Moreover, for any $m\in\calM$, $\set{\bfpsi_{m,n}}_{n\in\calN}$ is a UNTF for $\bbF^\calD$ and so~\eqref{eq.FP} gives $\norm{\bfPsi_m^*\bfPsi_m^{}}_\Fro^2=\tfrac{N^2}{D}$.
Subtracting these $M$ diagonal-block terms from the previous equation gives
\begin{align*}
\tfrac{M(M-1)N^2}{D}
&=\sum_{m\in\calM}\sum_{m'\neq m}\norm{\bfPsi_m^*\bfPsi_{m'}^{}}_\Fro^2\\
&\leq M(M-1)N^2
\max_{\substack{m\neq m'\\n,n'\in\calN}}\abs{\ip{\bfpsi_{m,n}}{\bfpsi_{m',n'}}}^2,
\end{align*}
where equality holds if and only if $\abs{\ip{\bfpsi_{m,n}}{\bfpsi_{m',n'}}}^2=\frac1D$ whenever $m\neq m'$.
Since for every $m\in\calM$ we further have that $\abs{\ip{\bfpsi_{m,n}}{\bfpsi_{m,n'}}}^2=\frac{N-D}{D(N-1)}\leq\frac1D$ for all $n\neq n'$,
this is actually a lower bound on the coherence of the concatenation of any $M$ $\ETF(D,N)$ for $\bbF^\calD$.
This motivates the following:

\begin{definition}
\label{def.MUETF}
Let $\calM$, $\calN$ and $\calD$ be sets of cardinality $M\geq1$ and $N\geq D\geq1$, respectively.
A sequence $\set{\bfpsi_{m,n}}_{m\in\calM,n\in\calN}$ of unit vectors in $\bbF^\calD$ is a \textit{mutually-unbiased-equiangular tight frame} (MUETF) for $\bbF^\calD$ if
\begin{equation}
\label{eq.def of MUETF}
\abs{\ip{\bfpsi_{m,n}}{\bfpsi_{m',n'}}}^2
=\left\{\begin{array}{cl}
\tfrac{N-D}{D(N-1)},&\ m=m', n\neq n',\smallskip\\
\frac1D,&\ m\neq m'.\end{array}\right.
\end{equation}
\end{definition}

We often denote an MUETF that consists of $M$ mutually unbiased $\ETF(D,N)$ as an ``$\MUETF(D,N,M)$".
In the special case where $N=D$, such an MUETF equates to a collection of $M$ \textit{mutually unbiased bases} (MUBs) for $\bbF^\calD$.
If instead $N=D+1$, this equates to $M$ \textit{mutually unbiased simplices} (MUSs) for $\bbF^\calD$~\cite{FickusS20,Schmitt19}.

We now derive an upper bound on the number $M$ of mutually unbiased $\ETF(D,N)$ that can exist.
Any unit vector $\bfphi$ in $\bbF^\calD$ ``lifts" to a rank-one orthogonal projection operator $\bfphi\bfphi^*$,
and the Frobenius inner product of any two such operators is
\begin{equation*}
\ip{\bfphi_{1}^{}\bfphi_{1}^*}{\bfphi_{2}^{}\bfphi_{2}^*}_{\Fro}
=\Tr(\bfphi_{1}^{}\bfphi_{1}^*\bfphi_{2}^{}\bfphi_{2}^*)
=\abs{\ip{\bfphi_{1}}{\bfphi_{2}}}^2.
\end{equation*}
In particular, if $\set{\bfpsi_{m,n}}_{m\in\calM,n\in\calN}$ is an MUETF for $\bbF^\calD$,
then the Gram matrix of the lifted vectors $\set{\bfpsi_{m,n}^{}\bfpsi_{m,n}^*}_{m\in\calM,n\in\calN}$
is the entrywise-modulus-squared $\abs{\bfPsi^*\bfPsi}^2$ of the Gram matrix $\bfPsi^*\bfPsi$ of $\set{\bfpsi_{m,n}}_{m\in\calM,n\in\calN}$.
Here, \eqref{eq.def of MUETF} implies that every off-diagonal block of  $\abs{\bfPsi^*\bfPsi}^2$ is $\frac1D\bfJ_\calN$
(where $\bfJ_\calN$ denotes an all-ones $\calN\times\calN$ matrix) and that every diagonal block of $\abs{\bfPsi^*\bfPsi}^2$ is $\tfrac{(D-1)N}{D(N-1)}\bfI_\calN+\tfrac{N-D}{D(N-1)}\bfJ_\calN$.
This equates to having
\begin{equation*}
\abs{\bfPsi^*\bfPsi}^2
=\tfrac{(D-1)N}{D(N-1)}\bfI
-[\tfrac{D-1}{D(N-1)}\bfI-\tfrac1D\bfJ_{\calM}]\otimes\bfJ_{\calN}.
\end{equation*}
Diagonalizing this matrix reveals that it has eigenvalues $\tfrac{(D-1)N}{D(N-1)}$, $0$ and $\tfrac{MN}{D}$ with multiplicities $M(N-1)$, $M-1$ and $1$, respectively.
In particular, $\abs{\bfPsi^*\bfPsi}^2$ has rank $M(N-1)+1$.
At the same time, the operators $\set{\bfpsi_{m,n}^{}\bfpsi_{m,n}^*}_{m\in\calM,n\in\calN}$ lie in the real inner product space of all self-adjoint $\calD\times\calD$ matrices,
meaning the rank of their Gram matrix $\abs{\bfPsi^*\bfPsi}^2$ is at most the dimension of this space, namely $D^2$ when the underlying field $\bbF$ is $\bbC$,
and $\frac12D(D+1)$ when $\bbF=\bbR$.
That is,
\begin{equation}
\label{eq.Gerzon}
M(N-1)+1\leq\left\{\begin{array}{cl}D^2,&\ \bbF=\bbC,\\\frac12D(D+1),&\ \bbF=\bbR. \end{array}\right.
\end{equation}
When $M=1$,
this reduces to a necessary condition on $\ETF(D,N)$ known as Gerzon's bound.
More generally, solving for $M$ in the above inequality gives the following result:

\begin{theorem}
\label{thm.Gerzon}
If an $\MUETF(D,N,M)$ exists and $N>1$ then
\begin{equation*}
M\leq\left\{\begin{array}{cl}
\lfloor\frac{D^2-1}{N-1}\rfloor,&\ \bbF=\bbC,\smallskip\\
\lfloor\frac{(D-1)(D+2)}{2(N-1)}\rfloor,&\ \bbF=\bbR.
\end{array}\right.
\end{equation*}
\end{theorem}

In the special case where $N=D$, this reduces to the classical upper bound on the maximal number of MUBs, namely that $M\leq N+1$ when $\bbF=\bbC$ and that
$M\leq\lfloor\frac D2\rfloor+1$ when $\bbF=\bbR$.
In the special case where $N=D+1$, this reduces to a recently derived upper bound on the maximal number of MUSs~\cite{Schmitt19}.
As detailed below, in at least these two special cases,
there are an infinite number of values of $D$ for which these bounds can be achieved.

\subsection{Constructing MUETFs from relative difference sets}

From Section~\ref{sec.background},
recall that restricting and then normalizing the characters of any abelian group $\calG$ of order $G$ to a nonempty $D$-element subset $\calD$ of $\calG$ yields a harmonic UNTF $\set{\bfpsi_\gamma}_{\gamma\in\hat{\calG}}$, \smash{$\bfpsi_\gamma(d):=D^{-\frac12}\gamma(d)$} for $\bbC^\calD$.
Further recall that for any subgroup $\calH$ of $\calG$ of order $H$,
such a subset $\calD$ is an $\calH$-$\RDS(N,H,D,\Lambda)$ (Definition~\ref{def.RDS}) if and only if this harmonic UNTF satisfies~\eqref{eq.RDS 3} where $N=\frac{G}{H}$.
Comparing~\eqref{eq.RDS 3} to~\eqref{eq.def of MUETF} immediately gives that such a harmonic UNTF yields a (harmonic) MUETF, where each individual ETF is indexed by a coset of $\calH^\perp$.
To formalize this connection,
we abuse notation, letting the indexing ``$\alpha\in\hat{\calG}/\calH^\perp$" denote letting $\alpha$ vary over any particular transversal (set of coset representatives) of $\calH^\perp$ with respect to $\hat{\calG}$.
Doing so permits us to uniquely factor any $\gamma\in\hat{\calG}$ as $\gamma=\alpha\beta$ where $\beta\in\calH^\perp$,
at which point comparing~\eqref{eq.RDS 3} to~\eqref{eq.def of MUETF} gives:

\begin{theorem}
\label{thm.RDS gives MUETF}
Letting $\calH$ and $\calD$ be a subgroup and nonempty subset of a finite abelian group $\calG$, respectively, the sequence
\begin{equation*}
\set{\bfpsi_{\alpha,\beta}}_{\alpha\in\hat{\calG}/\calH^\perp,\,\beta\in\calH^\perp}
\subseteq\bbC^\calD,
\quad\bfpsi_{\alpha,\beta}(d):=D^{-\frac12}\alpha(d)\beta(d),
\end{equation*}
is an $\MUETF(D,N,H)$ for $\bbC^\calD$ (Definition~\ref{def.MUETF}) if and only if $\calD$ is an $\calH$-$\RDS(N,H,D,\Lambda)$ for $\calG$ (Definition~\ref{def.RDS}).
\end{theorem}

In general, we refer to any MUETF created by Theorem~\ref{thm.RDS gives MUETF} as a \textit{harmonic} MUETF.
In the special case where $\calH=\set{0}$,
Theorem~\ref{thm.RDS gives MUETF} reduces to the known equivalence between harmonic ETFs and difference sets~\cite{XiaZG05,DingF07}.
Meanwhile, in the special case where $D=N$,
Theorem~\ref{thm.RDS gives MUETF} converts any $\calH$-$\RDS(D,H,D,\Lambda)$ into $H$ MUBs for $\bbC^\calD$ in a manner identical to that of~\cite{GodsilR09}.
If instead $N=D+1$,
Theorem~\ref{thm.RDS gives MUETF} yields $H$ MUSs for $\bbC^\calD$~\cite{FickusS20,Schmitt19}.
As we shall see, Theorem~\ref{thm.RDS gives MUETF} is a true generalization of these previously known results,
yielding infinite numbers of MUETFs that do not belong to any one of these three special categories.

Moving forward, it helps to note that if $\calD$ is any $\calH$-RDS for $\calG$,
then by a simple counting argument,
quotienting it by any subgroup $\calK$ of $\calH$ produces an $(\calH/\calK)$-RDS for $\calG/\calK$,
namely $\calD/\calK:=\set{d+\calK\in\calG/\calK: d\in\calD}$~\cite{Pott95}.
In particular, if $\calK$ has order $K$,
doing so transforms an $\RDS(N,H,D,\Lambda)$ into an $\RDS(N,\frac{H}{K},D,K\Lambda)$.
As an extreme case, quotienting an $\calH$-RDS by $\calK=\calH$ yields a difference set for $\calG/\calH$.
From the perspective of Theorem~\ref{thm.RDS gives MUETF},
these are special cases of more general ideas,
namely that any subcollection of $M$ mutually unbiased $\ETF(D,N)$ are still mutually unbiased, and that any single one of them is an ETF.

\begin{remark}
\label{rem.harmonic MUETFF}
On a related note,
the particular $\ETF(D,N)$ that arises from Theorem~\ref{thm.RDS gives MUETF} by taking $\alpha=1$ is identical to the harmonic ETF that arises from the difference set $\calD/\calH$ for $\calG/\calH$.
To elaborate,
for any $\beta\in\calH^\perp$ and $d\in\calD$,
\begin{equation}
\label{eq.RDS 4}
\bfpsi_{1,\beta}(d)
=D^{-\frac12}1(d)\beta(d)
=D^{-\frac12}\beta(d).
\end{equation}
At the same time, $\calH^\perp$ is naturally identified with the Pontryagin dual of $\calG/\calH$ via the isomorphism that maps any given $\gamma\in\calH^\perp$ to the character $g+\calH\mapsto\gamma(g)$.
Under this identification,
evaluating the $\beta$th member of the harmonic ETF arising from $\calD/\calH$ at $d+\calH$ gives~\eqref{eq.RDS 4}.
For this reason,
for any $\calH$-RDS $\calD$ for $\calG$,
we usually regard $\set{\bfpsi_\beta}_{\beta\in\calH^\perp}$, $\bfpsi_\beta:=\bfpsi_{1,\beta}$,
as the ``prototypical" ETF that arises from it.
Indeed, for any $\alpha\in\hat{\calG}$,
letting $\bfPsi_\alpha$ be the synthesis operator for $\set{\bfpsi_{\alpha,\beta}}_{\beta\in\calH^\perp}$,
we have $\bfPsi_\alpha=\bfDelta_\alpha\bfPsi$ where $\bfPsi:=\bfPsi_1$ is the synthesis operator of $\set{\bfpsi_\beta}_{\beta\in\calH^\perp}$ and where $\bfDelta_\alpha$ is the $\calD\times\calD$ unitary diagonal matrix whose $d$th diagonal entry is $\alpha(d)$:
\begin{equation*}
\bfPsi_\alpha(d,\beta)
=D^{-\frac12}\alpha(d)\beta(d)
=(\bfDelta_\alpha\bfpsi_\beta)(d)
=(\bfDelta_\alpha\bfPsi)(d,\beta).
\end{equation*}
\end{remark}

\subsection{Constructions of harmonic MUETFs}

As noted in~\cite{GodsilR09},
in the special case where $N=D$,
applying Theorem~\ref{thm.RDS gives MUETF} to an $\calH$-$\RDS(D,H,D,\Lambda)$ actually implies the existence of $H+1$ MUBs for $\bbC^\calD$: since $\abs{\bfpsi_{\alpha,\beta}(d)}=D^{-\frac12}$ for all $\alpha$, $\beta$ and $d$,
every orthonormal basis $\set{\bfpsi_{\alpha,\beta}}_{\beta\in\calH^\perp}$ is also unbiased to the standard basis.
This is especially significant since, for any prime power $Q$,
there is a classical construction of an $\RDS(Q,Q,Q,1)$~\cite{Pott95} which in turn yields $Q+1$ (the maximal number of) MUBs in $\bbC^Q$.
When $Q$ is odd, the construction is shockingly simple:
let
\begin{equation*}
\calG=\bbF_Q\times\bbF_Q,
\
\calH=\set{0}\times\bbF_Q,
\
\calD=\set{(x,x^2): x\in\bbF_Q},
\end{equation*}
where here and throughout, $\bbF_Q$ denotes the finite field of order $Q$.
Indeed, if $(x,x^2)-(y,y^2)\in\calH=\set{0}\times\bbF_Q$ then $x=y$ and so $(x,x^2)-(y,y^2)=(0,0)$.
Meanwhile, for any $(a,b)\in\calH^\rmc$ we have $a\neq0$,
and so there exists exactly one pair $(x,x^2),(y,y^2)$ such that
\begin{equation*}
(a,b)
=(x,x^2)-(y,y^2)
=(x-y,(x-y)(x+y)),
\end{equation*}
namely the pair arising from $x=\frac12(\frac ba+a)$ and $y=\frac12(\frac ba-a)$.
The construction is more complicated when $Q$ is even~\cite{Pott95}.
This is not surprising since in that case the characters of $\bbF_Q\times\bbF_Q$ are real-valued,
and there are at most $\frac Q2+1$ MUBs in $\bbR^Q$.

The proof of our main result (Theorem~\ref{thm.main result}) relies on mutually unbiased ETFs that are not MUBs,
and which arise from another classical RDS construction~\cite{Pott95}.
For any prime power $Q$ and $J\geq2$,
regard $\bbF_{Q^J}$ as a $J$-dimensional vector space over its subfield $\bbF_Q$.
Let $\tr:\bbF_{Q^J}\rightarrow\bbF_Q$, $\tr(x):=\sum_{j=0}^{J-1}x^{Q^j}$
be the field trace,
which is a nontrivial linear functional.
Let \smash{$\calG=\bbF_{Q^J}^\times$} be the (cyclic) multiplicative group of $\bbF_{Q^J}$,
let $\calH=\bbF_Q^\times$,
and consider the affine hyperplane
\begin{equation}
\label{eq.Bose RDS}
\calD=\set{x\in\bbF_{Q^J}^\times : \tr(x)=1},
\end{equation}
which has cardinality $Q^{J-1}$.
For any \smash{$y\in\bbF_{Q^J}^\times$},
\begin{equation}
\label{eq.RDS 5}
\calD\cap(y\calD)=\set{x\in\bbF_{Q^J}^\times : \tr(x)=1=\tr(y^{-1}x)}.
\end{equation}
For any $y\in\bbF_Q^\times$, $y\neq1$, the linearity of the trace gives $\tr(y^{-1}x)=y^{-1}\tr(x)$,
implying $\calD\cap(y\calD)$ is empty.
Meanwhile, for any $y\notin\bbF_Q^\times$, we can take $\set{1,y^{-1}}$ as the first two vectors in a basis $\set{z_j}_{j=1}^{J}$ for $\bbF_{Q^J}$ over $\bbF_Q$.
Since the trace is nontrivial,
the ``analysis operator" of this basis,
namely
\begin{equation*}
\bfL:\bbF_{Q^J}\rightarrow\bbF_Q^J,
\quad
\bfL(x)=(\tr(z_1x),\dotsc,\tr(z_Jx)),
\end{equation*}
has a trivial null space and is thus invertible.
Letting $\bfA$ be the $2\times J$ matrix whose rows are the first two rows of $\bfI$,
we thus have that the mapping $x\mapsto\bfA\bfL(x)=(\tr(x),\tr(y^{-1}x))$ has rank two.
This implies~\eqref{eq.RDS 5} is an affine subspace of codimension $2$
and so has cardinality $Q^{J-2}$.
Altogether,
\begin{equation*}
\#[\calD\cap(y\calD)]
=\left\{\begin{array}{cl}
Q^{J-1},&y=1,\smallskip\\
0,      &y\in\bbF_Q^\times,\ y\neq 1,\smallskip\\
Q^{J-2},&y\notin\bbF_Q^\times,
\end{array}\right.
\end{equation*}
and so $\calD$ is an $\RDS(\frac{Q^J-1}{Q-1},Q-1,Q^{J-1},Q^{J-2})$ for $\bbF_{Q^J}^\times$.

Quotienting this $\calH$-RDS by $\calH$ gives a difference set $\calD/\bbF_Q^\times$ for
$\bbF_{Q^J}^\times/\bbF_Q^\times$;
since the trace is linear over $\bbF_Q^\times$, it is
\begin{align*}
\calD/\bbF_Q^\times
=\set{x\bbF_Q^\times\in\bbF_{Q^J}^\times/\bbF_Q^\times : \tr(x)=1}\\
=\set{x\bbF_Q^\times\in\bbF_{Q^J}^\times/\bbF_Q^\times : \tr(x)\neq0},
\end{align*}
namely the complement of the well-known \textit{Singer} difference set for $\bbF_{Q^J}^\times/\bbF_Q^\times$,
defined as $\set{x\bbF_Q^\times\in\bbF_{Q^J}^\times/\bbF_Q^\times : \tr(x)=0}$.

Applying Theorem~\ref{thm.RDS gives MUETF} to this RDS yields $Q-1$ mutually unbiased $\ETF(Q^{J-1},\frac{Q^J-1}{Q-1})$.
To make this construction more explicit,
let $\alpha$ be a generator of $\bbF_{Q^J}^\times$,
and consider the isomorphism $\alpha^t\mapsto t$ from this group onto $\bbZ_{Q^J-1}$.
Under this isomorphism, $\calG$, $\calH$ and $\calD$ become
\begin{equation*}
\calG=\bbZ_{Q^J-1},\
\calH=\langle \tfrac{Q^J-1}{Q-1}\rangle,\
\calD=\set{d\in\bbZ_{Q^J-1}: \tr(\alpha^d)=1}.
\end{equation*}
This allows us to also regard
$\hat{\calG}$ as $\bbZ_{Q^J-1}$,
isomorphically identifying $s\in\bbZ_{Q^J-1}$ with the character
$t\mapsto\exp(\frac{2\pi\rmi st}{Q^J-1})$.
Under this identification, $\calH^\perp$ becomes
\begin{align*}
\calH^\perp
&=\set{s\in\bbZ_{Q^J-1}:
\exp(\tfrac{2\pi\rmi st}{Q^J-1})=1,\ \forall\, t\in\langle\tfrac{Q^J-1}{Q-1}\rangle}\\
&=\langle Q-1 \rangle.
\end{align*}
Any $s\in\bbZ_{Q^J-1}$ can be uniquely written as $s=m+(Q-1)n$ where $m$ lies in the transversal $\set{0,1,2,\dotsc,Q-2}$ of $\langle Q-1 \rangle$ and \smash{$n=0,\dotsc,\tfrac{Q^J-1}{Q-1}-1$}.
Under these identifications, the $n$th member of the $m$th mutually unbiased $\ETF(Q^{J-1},\frac{Q^J-1}{Q-1})$ for $\bbC^\calD$ produced by Theorem~\ref{thm.RDS gives MUETF} becomes
\begin{align*}
\bfpsi_{m,n}(d)
&=Q^{-\frac{J-1}2}\exp(\tfrac{2\pi\rmi[m+(Q-1)n]d}{Q^J-1})\\
&=\exp(\tfrac{2\pi\rmi md}{Q^J-1})
Q^{-\frac{J-1}2}\exp(\tfrac{2\pi\rmi(Q-1)nd}{Q^J-1}).
\end{align*}
Here, by Definition~\ref{def.MUETF},
$\abs{\ip{\bfpsi_{m,n}}{\bfpsi_{m',n'}}}^2$ has value $\frac1{Q^{J-1}}$ whenever $m\neq m'$, and otherwise has value
\begin{align*}
[\tfrac{Q^{J}-1}{Q-1}-Q^{J-1}][Q^{J-1}(\tfrac{Q^{J}-1}{Q-1}-1)]^{-1}
=\tfrac1{Q^J}.
\end{align*}
Combining these facts with the perspective of Remark~\ref{rem.harmonic MUETFF} gives:

\begin{corollary}
\label{cor.Singer MUETFF}
For any prime power $Q$ and any integer $J\geq2$,
let $\alpha$ be a generator of $\bbF_{Q^J}^\times$ and let
\begin{equation*}
\calD
=\biggset{d\in\bbZ_{Q^J-1}: \tr(\alpha^d)=\sum_{j=0}^{J-1}\alpha^{dQ^j}=1}
\subseteq\bbZ_{Q^J-1}.
\end{equation*}
Letting $\bfDelta$ be the $\calD\times\calD$ unitary diagonal matrix whose $d$th diagonal entry is $\exp(\tfrac{2\pi\rmi d}{Q^J-1})$,
\begin{equation*}
\set{\bfpsi_{m,n}}_{m=0,}^{Q-2}\,_{n=0}^{\frac{Q^J-1}{Q-1}-1}\subseteq\bbC^\calD,
\quad
\bfpsi_{m,n}:=\bfDelta^m\bfpsi_n,
\end{equation*}
is an $\MUETF(Q^{J-1},\frac{Q^J-1}{Q-1},Q-1)$ for $\bbC^\calD$ where
\begin{equation*}
\set{\bfpsi_n}_{n=0}^{\!\frac{Q^J-1}{Q-1}-1}\subseteq\bbC^\calD,
\quad
\bfpsi_n(d):=Q^{-\frac{J-1}2}\exp(\tfrac{2\pi\rmi(Q-1)nd}{Q^J-1}),
\end{equation*}
is the harmonic $\ETF(Q^{J-1},\frac{Q^J-1}{Q-1})$ for $\bbC^\calD$ arising from the complement of a Singer difference set.
In particular,
\begin{equation*}
\abs{\ip{\bfpsi_{m,n}}{\bfpsi_{m',n'}}}^2
=\left\{\begin{array}{cl}
\frac1{Q^J},&\ m=m',\, n\neq n',\smallskip\\
\frac1{Q^{J-1}},&\ m\neq m'.
\end{array}\right.
\end{equation*}
\end{corollary}

\begin{example}
\label{ex.MUETF(4,5,3)}
When $Q=4$ and $J=2$,
$X^4+X+1$ is a primitive polynomial in $\bbF_2[X]$,
meaning $\alpha=X+\langle X^4+X+1\rangle$ generates the multiplicative group of
\begin{align*}
\bbF_{16}
&=\bbF_2[X]/\langle X^4+X+1\rangle\\
&=\set{a+b\alpha+c\alpha^2+d\alpha^3: a,b,c,d\in\bbF_2,\ \alpha^4=\alpha+1}.
\end{align*}
As such, elements of $\bbF_{16}$ can be represented as the powers of the $2\times 2$ companion matrix $\bfA$ of $\alpha$, whose entries lie in $\bbF_2$:
\begin{equation*}
\bfA=\left[\begin{smallmatrix}
0&0&0&1\\
1&0&0&1\\
0&1&0&0\\
0&0&1&0
\end{smallmatrix}\right].
\end{equation*}
This representation facilitates the computation of
\begin{align*}
\calD
&=\set{d\in\bbZ_{15}: \tr(\alpha^d)=1}\\
&=\set{d\in\bbZ_{15}: \bfA^d+\bfA^{4d}=\bfI}\\
&=\set{1,2,8,4}.
\end{align*}
This $\RDS(5,3,4,1)$ yields three mutually unbiased $\ETF(4,5)$ with synthesis operators $\bfPsi$, $\bfDelta\bfPsi$, $\bfDelta^2\bfPsi$ where $\omega=\rme^{\frac{2\pi\rmi}{15}}$,
\begin{equation*}
\bfDelta=\left[\begin{array}{cccc}
\omega  &0&0&0\\
0&\omega^2&0&0\\
0&0&\omega^8&0\\
0&0&0&\omega^4
\end{array}\right],
\
\bfPsi=\frac12\left[\begin{array}{rrrrr}
1&\omega^{ 3}&\omega^{ 6}&\omega^{ 9}&\omega^{12}\\
1&\omega^{ 6}&\omega^{12}&\omega^{ 3}&\omega^{ 9}\\
1&\omega^{ 9}&\omega^{ 3}&\omega^{12}&\omega^{ 6}\\
1&\omega^{12}&\omega^{ 9}&\omega^{ 6}&\omega^{ 3}\\
\end{array}\right].
\end{equation*}
Here, $\bfPsi$ is also the synthesis operator of the harmonic ETF that arises from the difference set $\set{1,2,3,4}$ in $\bbZ_5$ that itself arises by quotienting the $\calD=\set{1,2,8,4}\subseteq\bbZ_{15}$ by $\calH=\langle 5\rangle$.
\end{example}

As seen in this example, in the special case where $J=2$,
Corollary~\ref{cor.Singer MUETFF} yields $Q-1$ mutually unbiased $\ETF(Q,Q+1)$,
namely $Q-1$ mutually unbiased $Q$-simplices.
Such MUSs recently arose~\cite{FickusS20} in a study of harmonic ETFs that are a disjoint union of regular simplices~\cite{FickusJKM18}.
In the next section, we reverse some of the analysis of~\cite{FickusJKM18,FickusS20}, using the MUETFs constructed in Corollary~\ref{cor.Singer MUETFF} to produce new ETFs.
Before doing so, note that by Theorem~\ref{thm.Gerzon}, the maximal number of complex mutually unbiased $\ETF(D,D+1)$ is at most
\begin{equation*}
\lfloor\tfrac{D^2-1}{(D+1)-1}\rfloor
=\lfloor D-\tfrac1D\rfloor
=D-1.
\end{equation*}
Altogether, we see that for any prime power $Q$, the maximal number of mutually unbiased complex $\ETF(Q,Q+1)$ is exactly $Q-1$.
Meanwhile, when $J\geq3$,
there is a gap between the number $Q-1$ of mutually unbiased complex $\ETF(Q^{J-1},\frac{Q^J-1}{Q-1})$ produced by Corollary~\ref{cor.Singer MUETFF}
and the upper bound on this number provided by Theorem~\ref{thm.Gerzon}, namely
\begin{equation*}
\lfloor\tfrac{D^2-1}{N-1}\rfloor
=\lfloor\tfrac{(Q-1)(Q^{J-1}+1)}{Q}\rfloor
=Q^{J-2}(Q-1).
\end{equation*}
This gap persists even if, following~\cite{GodsilR09},
we refine the analysis of~\eqref{eq.Gerzon} so as to account for the fact that the MUETF vectors $\bfpsi_{m,n}$ constructed by Corollary~\ref{cor.Singer MUETFF} have entries of  constant modulus $D^{-\frac12}$.
To elaborate, in this case each $\bfpsi_{m,n}$ lifts to an orthogonal projection operator $\bfpsi_{m,n}^{}\bfpsi_{m,n}^*$ which lies in real $(D^2-D+1)$-dimensional space of complex Hermitian $\calD\times\calD$ matrices with constant diagonal entries.
As such, \eqref{eq.Gerzon} refines to $M(N-1)+1\leq D^2-D+1$, namely
\begin{equation*}
M
\leq\lfloor\tfrac{D(D-1)}{N-1}\rfloor
=\Bigl\lfloor\tfrac{Q^{J-1}(Q^{J-1}-1)}{Q(\frac{Q^{J-1}-1}{Q-1})}\Bigr\rfloor
=Q^{J-2}(Q-1).
\end{equation*}

Though this might seem like an esoteric issue,
in the next section, we introduce a way of constructing new ETFs from MUETFs,
and this method very much depends on the number of mutually unbiased ETFs available.
And remarkably, the RDS literature itself warns that Corollary~\ref{cor.Singer MUETFF} can be improved upon: when $Q=4$ and $J=3$,
Corollary~\ref{cor.Singer MUETFF} yields $3$ mutually unbiased $\ETF(16,21)$ arising from an $\RDS(21,3,16,4)$ when in fact, an $\RDS(21,6,16,2)$ for $\bbZ_{126}$ exists~\cite{Pott95},
and applying Theorem~\ref{rem.harmonic MUETFF} to it yields $6$ mutually unbiased $\ETF(16,21)$.
We leave a deeper investigation of this issue for future work.

\section{New ETFs from MUETFs}
\label{sec.new ETFs}

In this section, we show how to combine the MUETFs of the previous section with certain other known ETFs to produce yet more ETFs.
Though MUETFs are an exotic ingredient, the recipe itself is simple:

\begin{theorem}
\label{thm.tensor}
If $\set{\bfphi_{n_1}}_{n_1\in\calN_1}$ is an $\ETF(D_1,N_1)$ for $\bbF^{\calD_1}$ and $\set{\bfpsi_{n_1,n_2}}_{n_1\in\calN_1,n_2\in\calN_2}$ is an $\MUETF(D_2,N_2,N_1)$ for $\bbF^{\calD_2}$,
and these parameters satisfy
\begin{equation}
\label{eq.consistency}
\tfrac{N_1-D_1}{D_1(N_1-1)}
=\tfrac{N_2-D_2}{N_2-1},
\end{equation}
then $\set{\bfphi_{n_1}\otimes\bfpsi_{n_1,n_2}}_{(n_1,n_2)\in\calN_1\times\calN_2}$ is an $\ETF(D_3,N_3)$ for $\bbF^{{\calD_1}\times{\calD_2}}$,
where $D_3=D_1D_2$, $N_3=N_1N_2$ and
\begin{align}
\label{eq.tensor 1}
N_3-D_3&=D_1N_1(N_2-D_2)+(N_1-D_1),\\[2pt]
\label{eq.tensor 2}
\tfrac{N_3-D_3}{D_3(N_3-1)}&=\tfrac{N_2-D_2}{D_2(N_2-1)},\\[2pt]
\label{eq.tensor 3}
\tfrac{D_3(N_3-D_3)}{N_3-1}&=D_1^2\tfrac{D_2(N_2-D_2)}{N_2-1}.
\end{align}
Moreover, if $\frac{D_2(N_2-D_2)}{N_2-1},\frac{(N_3-D_3)(N_3-1)}{D_3}\in\bbZ$
then $\frac{N_1}{D_1}\in\bbZ$.
\end{theorem}

\begin{proof}
Since $\set{\bfphi_{n_1}}_{n_1\in\calN_1}$ is an $\ETF(D_1,N_1)$,
\begin{equation*}
\abs{\ip{\bfphi_{n_1}}{\bfphi_{n_1'}}}^2
=\left\{\begin{array}{cl}
1,&n_1=n_1',\smallskip\\
\tfrac{N_1-D_1}{D_1(N_1-1)},& n_1\neq n_1'.
\end{array}\right.
\end{equation*}
Also $\set{\bfpsi_{n_1,n_2}}_{n_1\in\calN_1,n_2\in\calN_2}$ is an $\MUETF(D_2,N_2,N_1)$, and so Definition~\ref{def.MUETF} gives
\begin{equation*}
\abs{\ip{\bfpsi_{n_1,n_2}}{\bfpsi_{n_1',n_2'}}}^2
=\left\{\begin{array}{cl}
1,&(n_1,n_2)=(n_1',n_2'),\smallskip\\
\tfrac{N_2-D_2}{D_2(N_2-1)},&n_1=n_1', n_2\neq n_2',\smallskip\\
\frac1{D_2},&n_1\neq n_1'.\end{array}\right.
\end{equation*}
For any $(n_1,n_2),(n_1',n_2')\in\calN_1\times\calN_2$,
multiplying these expressions gives
\begin{multline}
\abs{\ip{\bfphi_{n_1}\otimes\bfpsi_{n_1,n_2}}{\bfphi_{n_1'}\otimes\bfpsi_{n_1',n_2'}}}^2\\
\label{eq.pf of tensor 1}
=\left\{\begin{array}{cl}
1,&(n_1,n_2)=(n_1',n_2'),\smallskip\\
\tfrac{N_2-D_2}{D_2(N_2-1)},&n_1=n_1', n_2\neq n_2',\smallskip\\
\tfrac{N_1-D_1}{D_2D_1(N_1-1)},&n_1\neq n_1'.\end{array}\right.
\end{multline}
In particular, \eqref{eq.consistency} equates to $\set{\bfphi_{n_1}\otimes\bfpsi_{n_1,n_2}}_{(n_1,n_2)\in\calN_1\times\calN_2}$ being equiangular.
To prove that this sequence of vectors is a tight frame for $\bbF^{\calD_1\times\calD_2}$
note that for any $n_1\in\calN_1$, the fact that $\set{\bfpsi_{n_1,n_2}}_{n_2\in\calN_2}$ achieves the Welch bound for $N_2$ vectors in $\bbF^{\calD_2}$ implies it is necessarily a UNTF for $\bbF^{\calD_2}$,
i.e., that its synthesis operator $\bfPsi_{n_1}$ satisfies
$\bfPsi_{n_1}^{}\bfPsi_{n_1}^*=\frac{N_2}{D_2}\bfI_{\calD_2}$.
Similarly, the synthesis operator $\bfPhi$ of $\set{\bfphi_{n_1}}_{n_1\in\calN_1}$ satisfies $\bfPhi\bfPhi^*=\frac{N_1}{D_1}\bfI_{\calD_1}$.
Thus, the frame operator of $\set{\bfphi_{n_1}\otimes\bfpsi_{n_1,n_2}}_{(n_1,n_2)\in\calN_1\times\calN_2}$ is
\begin{multline*}
\sum_{n_1\in\calN_1}\bfphi_{n_1}^{}\bfphi_{n_1}^*\otimes
\biggparen{\,\sum_{n_2\in\calN_2}\bfpsi_{n_1,n_2}^{}\bfpsi_{n_1,n_2}^*}\\
=\sum_{n_1\in\calN_1}\bfphi_{n_1}^{}\bfphi_{n_1}^*\otimes(\tfrac{N_2}{D_2}\bfI_{\calD_2})
=\tfrac{N_1N_2}{D_1D_2}\bfI_{{\calD_1}\times{\calD_2}}.
\end{multline*}
Thus,
$\set{\bfphi_{n_1}\otimes\bfpsi_{n_1,n_2}}_{(n_1,n_2)\in\calN_1\times\calN_2}$ is an $\ETF(D_3,N_3)$ for $\bbF^{\calD_1\times\calD_2}$ where $D_3=D_1D_2$ and $N_3=N_1N_2$.
By~\eqref{eq.pf of tensor 1},
the Welch bound of this $\ETF(D_3,N_3)$ equals that of our mutually unbiased $\ETF(D_2,N_2)$, giving~\eqref{eq.tensor 2}.
(Alternatively, \eqref{eq.tensor 2} follows directly from~\eqref{eq.consistency}.)
Next, multiplying~\eqref{eq.tensor 2} by $D_3(N_3-1)=D_1D_2(N_1N_2-1)$ and using~\eqref{eq.consistency} gives~\eqref{eq.tensor 1}:
\begin{align*}
N_3-D_3
&=\tfrac{D_1(N_1N_2-1)(N_2-D_2)}{N_2-1}\\
&=D_1N_1(N_2-D_2)+\tfrac{D_1(N_1-1)(N_2-D_2)}{N_2-1}\\
&=D_1N_1(N_2-D_2)+(N_1-D_1).
\end{align*}
Meanwhile, multiplying~\eqref{eq.tensor 2} by $D_3^2=D_1^2D_2^2$ immediately gives \eqref{eq.tensor 3}.
For the final conclusion, note that if
\begin{equation*}
\tfrac{D_2(N_2-D_2)}{N_2-1},\ \tfrac{(N_3-D_3)(N_3-1)}{D_3}
\end{equation*}
are integers then their product is as well;
by~\eqref{eq.tensor 2} and~\eqref{eq.tensor 1}, this product is
\begin{align*}
\tfrac{D_2(N_2-D_2)}{N_2-1}\tfrac{(N_3-D_3)(N_3-1)}{D_3}
&=\tfrac{D_2^2(N_3-D_3)}{D_3(N_3-1)}\tfrac{(N_3-D_3)(N_3-1)}{D_3}\\
&=(\tfrac{N_3-D_3}{D_1})^2\\
&=[N_1(N_2-D_2)+\tfrac{N_1}{D_1}-1]^2,
\end{align*}
and so $D_1$ necessarily divides $N_1$.
\end{proof}

\begin{example}
Since $(D_1,N_1)=(2,3)$ and $(D_2,N_2)=(4,5)$ satisfy~\eqref{eq.consistency},
i.e.,\ $\frac{3-2}{2(3-1)}=\frac14=\frac{5-4}{5-1}$,
we can apply Theorem~\ref{thm.tensor} to an $\ETF(2,3)$ and the $\MUETF(4,5,3)$ of Example~\ref{ex.MUETF(4,5,3)} to produce an $\ETF(8,15)$.
In particular, we take \smash{$\omega=\exp(\frac{2\pi\rmi}{15})$} as before,
and let $\set{\bfphi_n}_{n=0}^2$ be the $\ETF(2,3)$ with
\begin{equation*}
\bfPhi
=\left[\begin{array}{ccc}
\bfphi_0&\bfphi_1&\bfphi_2
\end{array}\right]
=\frac1{\sqrt{2}}\left[\begin{array}{ccc}
1&\omega^{ 5}&\omega^{10}\\
1&\omega^{10}&\omega^{ 5}
\end{array}\right].
\end{equation*}
Taking $\bfDelta$ and $\bfPsi$ as in Example~\ref{ex.MUETF(4,5,3)},
Theorem~\ref{thm.tensor} gives that
\begin{equation*}
\left[\begin{array}{ccc}
(\bfphi_0\otimes\bfPsi)
&(\bfphi_1\otimes\bfDelta\bfPsi)
&(\bfphi_2\otimes\bfDelta^2\bfPsi)
\end{array}\right]
\end{equation*}
is the synthesis operator of an $\ETF(8,15)$.

In fact, it turns out that perfectly shuffling the columns of this matrix---converting three collections of five vectors into five collections of three vectors---yields a harmonic ETF arising from a difference set for $\bbZ_{15}$ that is itself the sum of a difference set $\set{5,10}$ for the subgroup $\calH=\set{0,5,10}$ of $\bbZ_{15}$ and the $\calH$-RDS $\set{1,2,8,4}$ for $\bbZ_{15}$ from Example~\ref{ex.MUETF(4,5,3)},
namely
\begin{equation}
\label{eq.GMW 8x15}
\set{6,11,7,12,13,3,9,14}
=\set{5,10}+\set{1,2,8,4}.
\end{equation}
\end{example}

A little later on, we explain that while not every ETF produced by Theorem~\ref{thm.tensor} is harmonic,
this does occur infinitely often.
In such special cases, it turns out that the construction of Theorem~\ref{thm.tensor} is a reversal of some of the ideas from~\cite{FickusJKM18,FickusS20}.
To elaborate, \cite{FickusJKM18} considers ETFs that contain subsequences which are regular simplices for their span.
There, it was shown that certain Singer-complement-harmonic ETFs partition into such subsequences,
and moreover that their spans form a type of optimal packing of subspaces known as an \textit{equi-chordal tight fusion frame} (ECTFF).
(The techniques of~\cite{FickusJKM18} show, for example, that the harmonic $\ETF(8,15)$ arising from~\eqref{eq.GMW 8x15} partitions into three $\ETF(4,5)$, yielding an ECTFF that consists of three $4$-dimensional subspaces of $\bbC^8$.)
This analysis was refined in~\cite{FickusS20},
showing that these subspaces are actually \textit{equi-isoclinic} when the underlying difference set factors in a manner similar to~\eqref{eq.GMW 8x15}.

Theorem~\ref{thm.tensor} reverses this idea, concatenating $N_1$ subsequences, each of which is an $\ETF(D_2,N_2)$ for its span, to form an $\ETF(D_1D_2,N_1N_2)$.
Indeed, for any $n_1\in\calN_1$,
$\bfphi_{n_1}\otimes\bfPsi_{n_1}=(\bfphi_{n_1}\otimes\bfI)\bfPsi_{n_1}$ is the synthesis operator of \smash{$\set{\bfphi_{n_1}\otimes\bfpsi_{n_1,n_2}}_{n_2\in\calN_2}$},
meaning this subsequence is the embedding of the ETF \smash{$\set{\bfpsi_{n_1,n_2}}_{n_2\in\calN_2}$} into a $D_2$-dimensional subspace of $\bbF^{\calD_1\times\calD_2}$ via the isometry $\bfphi_{n_1}\otimes\bfI$.
In the special case where $N_2=D_2+1$,
the ETFs constructed by Theorem~\ref{thm.tensor} partition into a union of $N_1$ regular $D_2$-simplices.

We now apply Theorem~\ref{thm.tensor} with the only MUETFs we know of that are not MUBs,
namely those described in Corollary~\ref{cor.Singer MUETFF}:

\begin{theorem}
\label{thm.main result}
If an $\ETF(D,N)$ exists where $D<N<2D$ and $Q=\frac{D(N-1)}{N-D}$ is a power of a prime,
then for every positive integer $J$, there exists an $\ETF(D^{(J)},N^{(J)})$ where
\begin{align}
\label{eq.main 1}
D^{(J)}&=DQ^{J-1},\\
\label{eq.main 2}
N^{(J)}&=N(\tfrac{Q^J-1}{Q-1}),\\
\label{eq.main 3}
N^{(J)}-D^{(J)}&=DN(\tfrac{Q^{J-1}-1}{Q-1})+(N-D).
\end{align}
Here, for any $J\geq 2$,
\begin{align}
\label{eq.main 4}
\tfrac{D^{(J)}(N^{(J)}-1)}{N^{(J)}-D^{(J)}}&=Q^J,\\[2pt]
\label{eq.main 5}
\tfrac{(N^{(J)}-D^{(J)})(N^{(J)}-1)}{D^{(J)}}&=D^2Q^{J-2},\\[2pt]
\label{eq.main 6}
\tfrac{(N^{(J)}-D^{(J)})(N^{(J)}-1)}{D^{(J)}}&\notin\bbZ,\\[2pt]
\label{eq.main 7}
D^{(J)}+1&<N^{(J)}<2D^{(J)},
\end{align}
and so no $\ETF(D^{(J)},N^{(J)})$ can be real-valued.
\end{theorem}

\begin{proof}
When $J=1$, \eqref{eq.main 1}, \eqref{eq.main 2} and~\eqref{eq.main 3} become $D^{(1)}=D$, $N^{(1)}=N$ and $N^{(1)}-D^{(1)}=N-D$, respectively,
and the given $\ETF(D,N)$ is an $\ETF(D^{(1)},N^{(1)})$.
As such, assume $J\geq 2$.
Here, Corollary~\ref{cor.Singer MUETFF} provides an $\MUETF(Q^{J-1},\frac{Q^J-1}{Q-1},Q-1)$.
Moreover, since $N$ and $D$ are integers,
having $N<2D$ implies $N-D\leq D-1$ and so
\begin{equation*}
Q-1
=\tfrac{D(N-1)}{N-D}-1
=\tfrac{N(D-1)}{N-D}
\geq N.
\end{equation*}
Thus,
we can take just $N$ of these mutually unbiased ETFs to form an
\smash{$\MUETF(Q^{J-1},\frac{Q^J-1}{Q-1},N)$}.
To apply Theorem~\ref{thm.tensor} with this MUETF,
note that $(N_2,D_2)=(Q^{J-1},\frac{Q^J-1}{Q-1})$ satisfies
\begin{align}
\label{eq.pf of main 0}
N_2-D_2
&=\tfrac{Q^J-1}{Q-1}-Q^{J-1}
=\tfrac{Q^{J-1}-1}{Q-1},\\
\label{eq.pf of main 1}
N_2-1
&=\tfrac{Q^J-1}{Q-1}-1
=Q(\tfrac{Q^{J-1}-1}{Q-1}).
\end{align}
When combined with our assumption that $Q=\frac{D(N-1)}{N-D}$,
this implies that~\eqref{eq.consistency} is satisfied when $(N_1,D_1)=(D,N)$:
\begin{equation*}
\tfrac{N_2-D_2}{N_2-1}
=\tfrac1Q
=\tfrac{N-D}{D(N-1)}
=\tfrac{N_1-D_1}{D_1(N_1-1)}.
\end{equation*}
As such, Theorem~\ref{thm.tensor} yields an $\ETF(D^{(J)},N^{(J)})$ where
\begin{align*}
D^{(J)}&=D_3=D_1D_2=DQ^{J-1},\\
N^{(J)}&=N_3=N_1N_2=N(\tfrac{Q^J-1}{Q-1}),
\end{align*}
as claimed in~\eqref{eq.main 1} and~\eqref{eq.main 2}, respectively.
Moreover, combining~\eqref{eq.pf of main 0} with~\eqref{eq.tensor 1} immediately yields~\eqref{eq.main 3}.
Next, \eqref{eq.main 4} and~\eqref{eq.main 5} follow from combining \eqref{eq.tensor 2} and \eqref{eq.tensor 3} with \eqref{eq.pf of main 0}, \eqref{eq.pf of main 1} and the fact that $D_2=Q^{J-1}$.
Continuing, since $J\geq 2$,
\eqref{eq.main 3} gives
$N^{(J)}-D^{(J)}>N-D\geq1$,
namely one half of~\eqref{eq.main 7}.
For the remaining half,
note that since $Q=\frac{D(N-1)}{N-D}$ and $N<2D$,
\begin{equation*}
\tfrac{N^{(J)}}{D^{(J)}}
=\tfrac{N}{D}\tfrac{Q^J-1}{Q^{J-1}(Q-1)}
<\tfrac{N}{D}\tfrac{Q}{Q-1}
=\tfrac{N-1}{D-1}
\leq\tfrac{(2D-1)-1}{D-1}
=2.
\end{equation*}
Moreover, since $J\geq2$,
\begin{equation*}
\tfrac{D_2(N_2-D_2)}{N_2-1}=\tfrac{Q^{J-1}}{Q}=Q^{J-2}\in\bbZ,
\end{equation*}
while $1<\frac{N}{D}<2$ and so $\frac{N_1}{D_1}=\frac{N}{D}\notin\bbZ$.
As such, the final statement of Theorem~\ref{thm.tensor} gives~\eqref{eq.main 6}.
In particular, when $J\geq 2$,
\eqref{eq.main 6} and~\eqref{eq.main 7} imply that $(D^{(J)},N^{(J)})$ violates a well-known necessary condition on the existence of real ETFs~\cite{SustikTDH07}.
(We caution that~\eqref{eq.main 6} and~\eqref{eq.main 7} are not necessary when $J=1$.
In fact, some of the most fruitful applications of these ideas are scenarios where
$(D^{(1)},N^{(1)})=(D,N)$ either satisfies $\frac{(N-D)(N-1)}{D}\in\bbZ$ or $N=D+1$.)
\end{proof}

By our earlier remarks,
any ETF produced by Theorem~\ref{thm.main result} partitions into $N$ subsequences, each consisting of $\tfrac{Q^J-1}{Q-1}$ vectors that form an ETF for their $Q^{J-1}$-dimensional span.
In particular, taking $J=2$ yields ETFs that are disjoint unions of regular $Q$-simplices.
We also note that in light of~\eqref{eq.main 4} and~\eqref{eq.main 7},
Theorem~\ref{thm.main result} can be applied to any $\ETF(D^{(J)},N^{(J)})$ that it itself produces;
the interested reader can verify that doing so only yields a proper subset of those that arise by applying it to the original $\ETF(D,N)$.

By~\eqref{eq.main 5}, any ETF produced by Theorem~\ref{thm.main result} with $J\geq2$ satisfies the most basic necessary condition on the existence of harmonic ETFs,
namely that the \textit{index}
\begin{equation*}
D-\Lambda
=D-\tfrac{D(D-1)}{N-1}
=\tfrac{D(N-D)}{N-1}
\end{equation*}
of its underlying difference set is an integer.
At the same time, the fact that such an ETF satisfies~\eqref{eq.main 6} makes it unusual.
In fact, most known $\ETF(D,N)$ have the property that both $\frac{D(N-1)}{N-D}$ and $\frac{(N-D)(N-1)}{D}$ are integers,
including all real ETFs, SIC-POVMs, positive and negative ETFs, ETFs of redundancy two,
and all harmonic ETFs that arise from either Hadamard, McFarland, Spence, Davis-Jedwab and Chen difference sets~\cite{SustikTDH07,FickusJ19}.
Prior to Theorem~\ref{thm.main result},
the only known exceptions seemed to be certain harmonic ETFs (arising for example from  Singer, Paley, cyclotomic, Hall and twin-prime-power difference sets) and $\ETF(D,N)$ with $N=2D\pm1$~\cite{Renes07,Strohmer08}.
To be clear, Theorem~\ref{thm.main result} recovers some of these previously known unusual ETFs:
the next subsection is devoted to the special relationship between Theorem~\ref{thm.main result} and Singer difference sets;
moreover, applying Theorem~\ref{thm.main result} to any $\ETF(D,N)$ where $N=2D-1$ and $D$ is an even prime power yields another ETF of this same type.
However, in other cases, Theorem~\ref{thm.main result} yields ETFs with new parameters:

\begin{example}
\label{ex.ETF(6,9)}
The Naimark complement of a SIC-POVM for $\bbC^3$ is an $\ETF(6,9)$.
Since $Q=\frac{6(9-1)}{9-6}=16$ is a prime power,
Theorem~\ref{thm.main result} can be applied to it.
We walk through its proof in the special case where $J=2$.
Here, Corollary~\ref{cor.Singer MUETFF} provides $15$ mutually unbiased $\ETF(16,17)$ (regular $16$-simplices).
Since $\frac{9-6}{6(9-1)}=\frac1{16}=\frac{17-16}{17-1}$, condition~\eqref{eq.consistency} of Theorem~\ref{thm.tensor} is met,
and so taking tensor products of the members of this $\ETF(6,9)$ with the members of any $9$ of these $15$ mutually unbiased $\ETF(16,17)$ yields an $\ETF(D,N)$ with $(D,N)=(6(16),9(17))=(96,153)$.
This ETF is new~\cite{FickusM16}.
ETFs with these parameters cannot be real, SIC-POVMs, positive or negative since
$\frac{(N-D)(N-1)}{D}=\frac{361}{4}=(\frac{19}{2})^2$ is not an integer.
Moreover, no difference set of cardinality $96$ (or equivalently, $57$) exists in either $\bbZ_{3}\times\bbZ_3\times\bbZ_{17}$ or $\bbZ_9\times\bbZ_{17}$~\cite{Gordon19},
despite the fact that $\frac{D(N-D)}{N-1}=36$ is an integer.
\end{example}

\begin{example}
For another ``small" example, an $\ETF(10,16)$ exists~\cite{FickusM16} and $Q=\frac{10(16-1)}{16-10}=25$ is a prime power.
As such, when $J=2$ for example, Theorem~\ref{thm.main result} yields an $\ETF(D,N)$ with $(D,N)=(10(25),16(26))=(250,416)$.
This ETF is also new~\cite{FickusM16},
and $\frac{(N-D)(N-1)}{D}=(\frac{83}{5})^2$ is not an integer.
The existence of a difference set of cardinality 250 (or equivalently, 166) in a group of order $416$ is an open problem~\cite{Gordon19}.
\end{example}

In the next two subsections,
we evaluate the novelty of the ETFs produced by Theorem~\ref{thm.main result} in general.
It turns out that an infinite number of them are new,
and that another infinite number of them are not new.

\subsection{Recovering known ETFs with Theorem~\ref{thm.main result}}

Theorem~\ref{thm.main result} applies to any harmonic ETF  that arises from the complement of a Singer difference set, but doing so just yields another ETF of this same type.
To elaborate, when $K\geq 2$ and $(D,N)=(Q^{K-1},\frac{Q^K-1}{Q-1})$ for some prime power $Q$, we have
\begin{equation*}
1<\tfrac{N}{D}=\tfrac{Q^K-1}{Q^K-Q^{K-1}}<2,\quad
\tfrac{D(N-1)}{N-D}=Q^K.
\end{equation*}
Theorem~\ref{thm.main result} thus provides an $\ETF(D^{(J)},N^{(J)})$ with
\begin{align*}
D^{(J)}
&=D(Q^K)^{J-1}
=Q^{K-1}Q^{K(J-1)}
=Q^{JK-1},\\
N^{(J)}
&=N\tfrac{(Q^K)^J-1}{Q^K-1}
=\tfrac{Q^K-1}{Q-1}\tfrac{Q^{JK}-1}{Q^K-1}
=\tfrac{Q^{JK}-1}{Q-1}.
\end{align*}

As we now explain, this relates to the fact that Theorem~\ref{thm.main result} can be regarded as a generalization of a classical factorization of the complement of a Singer difference set due to Gordon, Mills and Welch~\cite{GordonMW62}.
Denoting the field trace from $\bbF_{Q^{JK}}$ to $\bbF_{Q^K}$ as ``$\tr_{JK:K}$,"
the freshman's dream gives
\begin{equation*}
\tr_{JK:1}(x)
=\tr_{K:1}(\tr_{JK:K}(x)),\quad\forall x\in\bbF_{Q^{JK}}.
\end{equation*}
We claim that $\calE_1=\calE_2\calE_3$ where
\begin{align*}
\calE_1&=\set{x_1\in\bbF_{Q^{JK}}^\times: \tr_{JK:1}(x_1)=1},\\
\calE_2&=\set{x_2\in\bbF_{Q^K}^{\times}: \tr_{K:1}(x_2)=1},\\
\calE_3&=\set{x_3\in\bbF_{Q^{JK}}^\times: \tr_{JK:K}(x_3)=1}.
\end{align*}
Indeed, for any $x_2\in\calE_2$ and $x_3\in\calE_3$,
\begin{equation*}
\tr_{JK:1}(x_2x_3)
=\tr_{K:1}(x_2\tr_{JK:K}(x_3))
=\tr_{K:1}(x_2)
=1,
\end{equation*}
and conversely, any $x_1\in\calE_1$ factors as $x_1=x_2x_3$ where $x_2:=\tr_{JK:K}(x_1)\in\calE_2$ and $x_3:=x_1^{}x_2^{-1}\in\calE_3$.
Since the hyperplanes $\calE_1$, $\calE_2$ and $\calE_3$ have cardinality $Q^{JK-1}$, $Q^{K-1}$ and $(Q^K)^{J-1}=Q^{JK-K}$, respectively, this ``$x_1=x_2x_3$" factorization is unique.
Applying the quotient homomorphism $x\mapsto\overline{x}:=x\bbF_{Q}^\times$ to  $\calE_1=\calE_2\calE_3$ gives $\calD_1=\calD_2\calD_3$ where
\begin{align}
\nonumber
\calD_1
&=\overline{\calE_1}
=\set{\overline{x}_1\in\bbF_{Q^{JK}}^\times/\bbF_{Q}^\times: \tr_{JK:1}(x_1)\neq0},\\
\label{eq.GMW 1}
\calD_2
&=\overline{\calE_2}
=\set{\overline{x}_2\in\bbF_{Q^K}^{\times}/\bbF_{Q}^\times: \tr_{K:1}(x_2)\neq0},\\
\nonumber
\calD_3
&=\overline{\calE_3}
=\set{\overline{x}_3\in\bbF_{Q^{JK}}^\times/\bbF_{Q}^\times:\tr_{JK:K}(x_3)\in\bbF_{Q}^\times}.
\end{align}
Here, $\calD_1$ and $\calD_2$ are the complements of Singer difference sets in $\bbF_{Q^{JK}}^\times/\bbF_{Q}^\times$ and $\bbF_{Q^K}^{\times}/\bbF_{Q}^\times$, respectively,
while $\calD_3$ is an
\begin{equation*}
\RDS(\tfrac{Q^{JK}-1}{Q^K-1},\tfrac{Q^K-1}{Q-1},Q^{(J-1)K},(Q-1)Q^{(J-2)K})
\end{equation*}
obtained by quotienting the RDS $\calE_3$ (of type~\eqref{eq.Bose RDS} where ``$Q$" is $Q^K$) by \smash{$\bbF_Q^\times$}.
When written additively, we have $\calD_1=\calD_2+\calD_3$ where $\calD_1$, $\calD_2$ and $\calD_3$ are subsets of $\bbZ_{N_1}$ where $N_1=\tfrac{Q^{JK}-1}{Q-1}$.
In fact, \eqref{eq.GMW 1} gives that $\calD_2$ is a subset of the subgroup of $\bbZ_{N_1}$ of order $N_2=\frac{Q^K-1}{Q-1}$, namely $\langle N_3\rangle$ where
$N_3=\frac{N_1}{N_2}=\frac{Q^{JK}-1}{Q^K-1}$.
As such, any $d_1\in\calD_1$ can be written as $d_1=N_3d_2+d_3$ where $N_3d_2\in\calD_2$, $d_3\in\calD_3$.
This in turn implies that the value of any given character of $\bbZ_{N_1}$ at  $d_1\in\calD_1$ is
\begin{align*}
\exp(\tfrac{2\pi\rmi nd_1}{N_1})
&=\exp(\tfrac{2\pi\rmi n(N_3d_2+d_3)}{N_1})\\
&=\exp(\tfrac{2\pi\rmi nd_2}{N_2})\exp(\tfrac{2\pi\rmi nd_3}{N_3}).
\end{align*}
From this, we see that each vector in the harmonic ETF arising from $\calD_1$ is a tensor product of a vector in the harmonic ETF arising from $\calD_2$ with a vector in the harmonic tight frame arising from the RDS $\calD_3$.

Overall, we see that the classical factorization of the complements of certain Singer difference sets~\cite{GordonMW62} indeed leads to a special case of the ``ETF-tensor-MUETF" construction of Theorem~\ref{thm.tensor} where, as in the proof of Theorem~\ref{thm.main result}, the MUETF in question arises from an RDS.
In particular, the harmonic ETF that arises from the complement of the Singer difference set in $\bbF_{Q^{JK}}^\times/\bbF_{Q}^\times$ partitions into \smash{$\frac{Q^K-1}{Q-1}$} copies of the harmonic ETF that arises from the complement of the Singer difference set in \smash{$\bbF_{Q^{JK}}^{\times}/\bbF_{Q^K}^\times$},
each isometrically embedded into a $Q^{(J-1)K}$-dimensional subspace of a common space of dimension $Q^{JK-1}$.
In the special case where $J=2$, this partitions a harmonic $\ETF(Q^{2K-1},\frac{Q^{2K}-1}{Q-1})$ into \smash{$\frac{Q^K-1}{Q-1}$} embedded regular $Q^K$-simplices,
recovering a result of~\cite{FickusJKM18}.

\subsection{Constructing new ETFs with Theorem~\ref{thm.main result}}

Theorem~\ref{thm.main result} requires an $\ETF(D,N)$ where
\begin{equation}
\label{eq.necessary}
\tfrac{D(N-1)}{N-D}\ \text{is a prime power and}\ D<N<2D.
\end{equation}
Some examples of such ETFs are regular simplices, that is, have $N=D+1$.
All other examples have $D+1<N<2D$,
meaning the parameters $(N-D,N)$ of its Naimark complement satisfy $N>2(N-D)>2$.
As discussed in Section~\ref{sec.background},
this means such an $\ETF(D,N)$ is either the Naimark complement of a SIC-POVM,
satisfies $N=2D-1$ where $D$ is even,
is a harmonic ETF,
or is the Naimark complement of a positive and/or negative ETF
(Definition~\ref{def.pos neg ETFs}).

In light of the previous subsection,
we ignore $\ETF(D,N)$ whose parameters match those of one that arises from the complement of a Singer difference set, as applying Theorem~\ref{thm.main result} to them only recovers ETFs with known parameters.
With a little work, one finds that this includes all $\ETF(D,N)$ that satisfy~\eqref{eq.necessary} and are either regular simplices, have $N=2D-1$ where $D$ is even, or are harmonic ETFs arising from the complements of difference sets of the following types: Singer, Paley, cyclotomic, Hall and twin-prime-power.
Moreover, every $\ETF(D,N)$ that satisfies~\eqref{eq.necessary} and whose Naimark complement is a SIC-POVM is an $\ETF(6,9)$, and every $\ETF(3,9)$ is both $2$-positive and $6$-negative.
Some harmonic ETFs are also positive or negative, including those that arise from Hadamard, McFarland, Spence, and Davis-Jedwab difference sets~\cite{FickusJ19}.

Our search thus reduces to the following:
find $\ETF(D,N)$ that satisfy~\eqref{eq.necessary} and are either Naimark complements of positive or negative ETFs or are harmonic ETFs that arise from (complements of) difference sets which are not one of the aforementioned types.
In the latter case, after searching the literature~\cite{JungnickelPS07},
the only potential candidates that we found are difference sets due to Chen~\cite{Chen97},
whose complements yield $\ETF(D,N)$ where
\begin{equation*}
D = Q^{2J-1}[2(\tfrac{Q^{2J}-1}{Q-1})-1],
\quad
N = 4Q^{2J}(\tfrac{Q^{2J}-1}{Q-1}),
\end{equation*}
where $J\geq 2$ and $Q$ is either a power of $3$ or an even power of an odd prime.
Here, $D<N<2D$ and
\begin{equation*}
\tfrac{D(N-1)}{N-D}=[2(\tfrac{Q^{2J}-1}{Q-1})-1]^2,
\end{equation*}
is sometimes a prime power and sometimes is not.
In cases where it is, applying Theorem~\ref{thm.main result} to the corresponding $\ETF(D,N)$ seems to yield new ETFs with enormous parameters;
we leave a deeper investigation of them for future work.

The only known ETFs that remain to be considered are $\ETF(D,N)$ whose Naimark complements are either positive or negative.
Taking the complementary parameters of those in~\eqref{eq.pos neg D} and~\eqref{eq.pos neg N},
this means there exists $L\in\set{1,-1}$ and integers $K\geq1$ and $S\geq2$ such that
\begin{align*}
D&=[S(\tfrac{K-1}{K})+L][S(K-1)+L],\\
N&=(S+L)[S(K-1)+L].
\end{align*}
As noted in~\cite{FickusJ19}, such ETFs satisfy $D<N<2D$ provided we exclude $1$-positive ETFs, $2$-negative ETFs, and ETFs of type $(3,-1,2)$ or $(3,-1,3)$.
That is, in terms of the partial list (a)--(g) of known families of  positive and negative ETFs given in Subsection~\ref{subsec.pos neg},
we exclude all ETFs from (a) and (d), the first two members of (e), and the first member of (f).
The remaining ETFs on this list satisfy~\eqref{eq.necessary} if and only if
\begin{align*}
\tfrac{D(N-1)}{N-D}
&=\tfrac{K}{S}[S(\tfrac{K-1}{K})+L][S^2(K-1)+KLS]\\
&=[S(K-1)+KL]^2
\end{align*}
is a prime power,
namely if and only if
\begin{equation}
\label{eq.pos neg necessary}
Q=S(K-1)+KL\ \text{is a prime power.}
\end{equation}
Here, $Q\equiv L\bmod(K-1)$, and substituting \smash{$S=\frac{Q-KL}{K-1}$} into the above expressions for $D$, $N$ and \smash{$\tfrac{D(N-1)}{N-D}$} gives
\begin{align*}
D&=\tfrac{Q}{K}[Q-(K-1)L],\\
N&=\tfrac{Q-L}{K-1}[Q-(K-1)L],\\
\tfrac{D(N-1)}{N-D}&=Q^2.
\end{align*}

Many ETFs from (b) satisfy~\eqref{eq.pos neg necessary}.
When $K=2,3,4,5$ for example,
we only need prime powers $Q\equiv 1\bmod(K-1)$ such that $S=\frac{Q-K}{K-1}$ satisfies $S\geq K$ and $S\equiv 0,1\bmod K$,
that is, such that $Q\geq K^2$ and $Q\equiv K,2K-1\bmod K(K-1)$.
An infinite number of such $Q$ exist, including all powers of $K$.
More generally, for any $K\geq 2$, we need prime powers $Q\equiv 1\bmod(K-1)$ such that $S=\frac{Q-K}{K-1}$ is sufficiently large and has the property that $K\mid S(S-1)$.
An infinite number of such $Q$ exist:
in fact, since $K$ and $K-1$ are relatively prime, their product is relatively prime to their sum, at which point Dirichlet's theorem implies there are an infinite number of primes $Q$ such that $Q\equiv 2K-1\bmod K(K-1)$, implying $S=\frac{Q-K}{K-1}\equiv 1\bmod K$.

All ETFs from (c) satisfy~\eqref{eq.pos neg necessary}:
when $(K,L,S)=(P,1,P)$ for some prime power $P$,
$Q=S(K-1)+KL=P^2$ is also a prime power.
(In contrast, Steiner ETFs arising from projective planes of order $P$ are of type $(P+1,1,P+1)$~\cite{FickusJ19}, and in this case $Q=(P+1)^2$ is only sometimes a prime power, such as when $P$ is a Mersenne prime.)

An infinite number of the (nonexcluded) ETFs from (e) satisfy~\eqref{eq.pos neg necessary}:
having $(K,L,S)=(3,-1,S)$ where $S\geq5$, $S\equiv0,2\bmod 3$ implies
$Q=S(K-1)+KL=2S-3$ can be any prime power such that $Q\geq 7$, $Q \equiv 1,3 \bmod 6$,
including all powers of $3$ apart from $3$ itself.

None of the (nonexcluded) ETFs from (f) satisfy~\eqref{eq.pos neg necessary}:
when $(K,L,S)=(2^J+1,-1,2^J+1)$ for some $J\geq 2$,
$Q=S(K-1)+KL=(2^J+1)(2^J-1)$ is not a prime power.

An infinite number of the ETFs from (f) satisfy~\eqref{eq.pos neg necessary}:
when $(K,L,S)=(4,-1,S)$ where $S\equiv3\bmod8$, we have that $Q=S(K-1)+KL=3S-4$ can be any prime power $Q$ such that $Q \equiv 5 \bmod 24$.

Applying Theorem~\ref{thm.main result} to the Naimark complements of these positive and negative ETFs yields the following result:

\begin{corollary}
\label{cor.pos neg}
Let $K\geq 2$, $L\in\set{1,-1}$, and let $Q$ be any prime power such that either:
\begin{enumerate}
\renewcommand{\labelenumi}{(\roman{enumi})}
\item
$Q\geq K^2$ and $Q\equiv K,2K-1\bmod K(K-1)$ where $K=2,3,4,5$ and $L=1$,
\item
$Q=K$ where $L=1$,
\item
$Q\geq 7$ and $Q \equiv 1,3 \bmod 6$ where $K=3$ and $L=-1$,
\item
$Q \equiv 5 \bmod 24$ where $K=4$ and $L=-1$.
\end{enumerate}
Then for any $J\geq 1$,
an $\ETF(D^{(J)},N^{(J)})$ exists where
\begin{align*}
D^{(J)}
&=\tfrac1K[Q-(K-1)L]Q^{2J-1},\\[2pt]
N^{(J)}
&=\tfrac{Q-(K-1)L}{K-1}(\tfrac{Q^{2J}-1}{Q+L}),\\[2pt]
N^{(J)}-D^{(J)}
&=\tfrac{Q^{2J-1}[Q-(K-1)L]^2-K[Q-(K-1)L]}{K(K-1)(Q+L)}.
\end{align*}
\end{corollary}

In the special case of (i) where $K=2$,
Corollary~\ref{cor.pos neg} yields $\ETF(D^{(J)},N^{(J)})$ whose Naimark complements have parameters~\eqref{eq.new ETF from 2 pos} for any prime power $Q\geq4$ and $J\geq 1$.
These ETFs arise by applying Theorem~\ref{thm.main result} to the Naimark complements of Steiner ETFs that themselves arise from BIBDs consisting of all $2$-element subsets of a $(Q-1)$-element vertex set~\cite{FickusMT12,FickusJ19}.
Meanwhile, case (ii) yields $\ETF(D^{(J)},N^{(J)})$ whose Naimark complements have parameters~\eqref{eq.neq ETF from GQ} for any prime power $Q$.

Remarkably, all of the $\ETF(D^{(J)},N^{(J)})$ produced by Corollary~\ref{cor.pos neg} with $J\geq2$ seem to be new:
though~\eqref{eq.main 5} implies that they satisfy one necessary condition of harmonic ETFs,
some of them are not harmonic (Example~\ref{ex.ETF(6,9)}),
and none of them seem to arise from known difference sets~\cite{JungnickelPS07};
since these ETFs satisfy~\eqref{eq.main 6},
we know that neither they nor their Naimark complements are real, SIC-POVMs, positive or negative, or have redundancy two.

\section{Conclusions and Future Work}

Theorem~\ref{thm.tensor} produces an ETF by taking certain tensor products of the vectors in a given ETF and MUETF with compatible parameters.
Theorem~\ref{thm.main result} is the special case of this idea in which the MUETF arises from a certain classical RDS.
It is a generalization of the classical Gordon-Mills-Welch factorization of the complement of a Singer difference set~\cite{GordonMW62},
and a reversal of some of the analysis of~\cite{FickusJKM18,FickusS20}.
When applied to various families of known positive and negative ETFs~\cite{FickusJ19},
Theorem~\ref{thm.main result} yields many infinite families of new ETFs, some of which are summarized in Corollary~\ref{cor.pos neg}.

Of course, it would be nice to be able to apply Theorem~\ref{thm.tensor} more broadly.
Doing so requires a better fundamental understanding of MUETFs.
Since harmonic MUETFs equate to RDSs (Theorem~\ref{thm.RDS gives MUETF}),
one approach is to devote more time and energy to their study;
see~\cite{Pott95} for some interesting, important open problems concerning RDSs.
More generally, for what $(D,N,M)$ does an $\MUETF(D,N,M)$ exist?
In the special case where $D=N$, MUETFs reduce to MUBs,
and though some MUBs arise from RDSs~\cite{GodsilR09}, not all seemingly do:
some arise from tensor products of other MUBs,
and yet others arise from mutually orthogonal Latin squares~\cite{WocjanB05}.
Moreover, some MUBs are real~\cite{BoykinSTW05}.
To what extent do these facts generalize to the non-MUB case?
We leave these questions for future work.

\section*{Acknowledgments}
\noindent
The views expressed in this article are those of the authors and do not reflect the official policy or position of the United States Air Force, Department of Defense, or the U.S.~Government.


\end{document}